\documentclass[11pt,parskip=half]{article}

\usepackage{geometry}                
\usepackage[parfill]{parskip}    

\usepackage[utf8]{inputenc}

\usepackage{graphicx}

\usepackage{amsmath, amssymb, amsthm, amsfonts}
\usepackage{mathtools}
\usepackage{bbm} 

\usepackage{hyperref}

\usepackage[noabbrev,capitalise]{cleveref}

\usepackage[svgnames]{xcolor}
\usepackage[breakable,theorems]{tcolorbox}
\usepackage[font=small,labelfont=bf]{caption}

\definecolor{DarkBlue}{HTML}{2c7f90}
\definecolor{DarkBrown}{HTML}{88632a}

\usepackage{libertine}
\usepackage[libertine]{newtxmath}
\usepackage{inconsolata}
\usepackage{microtype}

\usepackage{enumitem}

\usepackage{booktabs}

\usepackage{csquotes}

\usepackage{titlesec}
\titlespacing*{\paragraph}{0pt}{.6\baselineskip plus .3\baselineskip minus .2\baselineskip}{1em}

\usepackage{pgf}

\usepackage[backend=biber,
            giveninits=true,
            mincitenames=1,
            maxcitenames=2,
            maxbibnames=99,
            uniquename=false,
            uniquelist=false,
            style=ext-authoryear,
            url = false,
            eprint = false,
            isbn = false,
            doi = false,
            articlein = false,
            dashed = false,
           ]{biblatex}

\DeclareCiteCommand{\citeyear}
    {}
    {\bibhyperref{\printdate}}
    {\multicitedelim}
    {}

\newtcbtheorem[crefname={Remark}{Remark}]{remark}{Remark}{%
  theorem style=plain, 
  coltitle=white!20!black,
  colframe=white,
  colback=black!0!white,
  breakable
  }{rem}

\newtcbtheorem[crefname={Definition}{Definition}]{definition}{Definition}{%
  theorem style=plain, 
  coltitle=Violet!45!black,
  colframe=white,
  colback=green!0!white,
  breakable
  }{def}

\newtcbtheorem[crefname={Lemma}{Lemma}]{lemma}{Lemma}{%
  theorem style=plain, 
  coltitle=green!30!black,
  colframe=white,
  colback=green!0!white,
  breakable
  }{lem}

\newtcbtheorem[crefname={Theorem}{Theorem}]{theorem}{Theorem}{%
  theorem style=plain, 
  coltitle=orange!45!black,
  colframe=white,
  colback=orange!0!white,
  breakable
  }{thm}

\newtcbtheorem[crefname={Corollary}{Corollary}]{corollary}{Corollary}{%
  theorem style=plain, 
  coltitle=orange!45!black,
  colframe=white,
  colback=orange!0!white,
  breakable
  }{cor}

\newtcbtheorem[crefname={Proposition}{Proposition}]{proposition}{Proposition}{%
  theorem style=plain, 
  coltitle=blue!35!black,
  colframe=white,
  colback=blue!3!white,
  breakable
  }{prop}

\newtcbtheorem[crefname={Example}{Example}]{example}{Example}{%
  theorem style=plain, 
  coltitle=white!20!black,
  colframe=white,
  colback=white,
  breakable
  }{ex}

\newcommand{\PC}{\mathcal{P}}

\newcommand{\RR}{\mathbb{R}}
\newcommand{\NN}{\mathbb{N}}

\newcommand{\interior}{\mathrm{int}}
\newcommand{\closure}{\mathrm{cl}}

\newcommand{\X}{X}
\newcommand{\Y}{Y}
\newcommand{\XY}{{\X\times\Y}}
\newcommand{\Xres}{{\tilde\X}}
\newcommand{\Yres}{{\tilde\Y}}
\newcommand{\XYres}{{\Xres\times\Yres}}
\newcommand{\cres}{{\tilde{c}}}
\newcommand{\fres}{{\tilde{f}}}

\newcommand{\OT}{T}
\newcommand{\CC}{{S}}
\newcommand{\CCc}{{\CC_c}}
\newcommand{\CCcc}{\smash{\CC^c_{\raisebox{1.25ex}{$\scriptstyle c$}}}}
\newcommand{\RRext}{\RR\cup\{-\infty\}}

\newcommand{\pr}{p} 

\newcommand{\ind}{\mathbbm{1}}

\newcommand{\domain}{\mathrm{domain}}
\newcommand{\range}{\mathrm{range}}

\newcommand{\dif}{\,\mathrm{d}}
\newcommand{\supp}{\mathrm{supp}}

\newcommand{\ot}{T}
\newcommand{\otmap}{t}
\newcommand{\couple}{\mathcal{C}}

\newcommand{\osupp}{\Gamma}

\bibliography{sources}
\graphicspath{{figures/}}

\title{On the Uniqueness of Kantorovich Potentials}

\author{ 
  Thomas Staudt%
  \thanks{Institute for Mathematical Stochastics, University of Göttingen}
  \thanks{Cluster of Excellence: Multiscale Bioimaging (MBExC), University Medical Center, Göttingen} \\[-0.25ex]
  {\small \href{mailto:thomas.staudt@uni-goettingen.de}{thomas.staudt@uni-goettingen.de}} \vspace{1.2ex}\\
  Shayan Hundrieser%
  \footnotemark[1] \footnotemark[2] \\[-0.25ex]
  {\small \href{mailto:s.hundrieser@math.uni-goettingen.de}{s.hundrieser@math.uni-goettingen.de}} \vspace{1.2ex}\\
  Axel Munk%
  \footnotemark[1] \footnotemark[2] \,\thanks{Max Planck Institute for Multidisciplinary Sciences, Göttingen} \\[-0.25ex]
  {\small \href{mailto:munk@math.uni-goettingen.de}{munk@math.uni-goettingen.de}}
}

\date{\vspace{-3ex}}

\begin{document}

\maketitle

\begin{abstract}
  \noindent
  Kantorovich potentials denote the dual solutions of the renowned optimal transportation problem.
  Uniqueness of these solutions is relevant from both a theoretical and an algorithmic point of view, and has recently emerged as a necessary condition for asymptotic results in the context of statistical and entropic optimal transport.
  In this work, we challenge the common perception that uniqueness in continuous settings is reliant on the connectedness of the support of at least one of the involved measures, and we provide mild sufficient conditions for uniqueness even when both measures have disconnected support.
  Since our main finding builds upon the uniqueness of Kantorovich potentials on connected components, we revisit the corresponding arguments and provide generalizations of well-known results.
  To this end, we introduce the notion of induced regularity and employ it to extend the regularity theory of Kantorovich potentials advanced by \textcite{gangbo1996geometry} to more general cost functions in $\RR^d$ and to geodesic spaces.
  
\end{abstract}



\section{Introduction}

Optimal transport theory addresses the question how mass can be moved from a source to a target distribution in the most cost-efficient way.
While the history of this mathematical quest is long and rich, dating back to \textcite{monge} and then \textcite{kantorovich}, who framed the theory in its modern formulation, it has drawn an enormous amount of attention in the past decades.
In-depth monographs cover analytical, probabilistic, and geometric \parencite{rachev1998massI,villani2008optimal,santambrogio2015optimal,ambrosio2021lectures,figalli2021an} as well as computational \parencite{peyre2019computational} and statistical \parencite{panaretos2020invitation} perspectives, while countless applications span from economics \parencite{galichon2016optimal} and biology \parencite{Schiebinger19,tameling2021colocalization, naas2024multimatch} to machine learning (see \cite{montesuma2023recent} for a survey).

For given probability distributions $\mu\in\PC(\X)$ and $\nu\in\PC(\Y)$ on measurable spaces $X$ and $Y$, the optimal transport problem is to find the minimal transportation cost
\begin{subequations}\label{eq:ot}
\begin{equation}\label{eq:otprimal}
  \ot_c(\mu, \nu) = \inf_{\pi\in\couple(\mu, \nu)} \int\! c\dif\pi,
\end{equation}
where $c\colon\XY\to\RR_+$ denotes a non-negative measurable \emph{cost function} that quantifies the effort of moving one unit of mass between elements in $\X$ and $\Y$, and where $\couple(\mu, \nu)\subset\PC(\XY)$ is the set of probability measures with marginal distributions $\mu$ and $\nu$.
Any solution $\pi$ to \eqref{eq:otprimal} is called an \emph{optimal transport plan}.
Conditions under which optimal transport plans exist are well-known, see for example \textcite{villani2008optimal} for the general framework of Polish spaces and lower-semicontinuous cost functions.
Under these assumptions, the optimal transport problem~\eqref{eq:otprimal} also admits a dual formulation that reliably serves as a fertile ground for investigating structural properties of $\ot_c$.
In fact, it holds that 
\begin{equation}\label{eq:otdual}
  \ot_c(\mu, \nu)
  =
  \sup_{f\in L_1(\mu)} \int\! f \dif\mu + \int\! f^c\dif\nu,
\end{equation}%
\end{subequations}
where $f^c$, the \emph{$c$-transform} of $f$, denotes the largest function satisfying $f(x) + f^c(y) \le c(x, y)$ for all $x\in\X$ and $y\in\Y$.
Specific optimizers $f$ for the dual problem -- namely, those which can be written as the $c$-transform of a function $g$ on $\Y$ -- are called \emph{Kantorovich potentials}, which exist under mild conditions \parencite[Theorem~5.10]{villani2008optimal}.

Recently, advances in statistical and entropic optimal transport, i.e., when \eqref{eq:otprimal} is regularized by an entropy penalty, have highlighted the need for conditions that ensure the \emph{uniqueness} of said potentials (see \cref{sec:related} for a literature overview).
Uniqueness here means that any two Kantorovich potentials $f$ for a given optimal transport problem only differ by a constant $\mu$-almost surely.
This is equivalent to $\nu$-almost sure uniqueness of the $c$-transformed potentials $f^c$ (see \cref{lem:uniqueness}).
Practical guarantees for uniqueness have mainly been derived in continuous settings, where $\mu$, $\nu$ and $c$ are sufficiently regular.
A universal requirement is a connected support of at least one of the measures $\mu$ or $\nu$.
This requirement emerges naturally, since the standard proof technique -- concluding uniqueness of the potential from uniqueness of its gradients -- cannot bridge separated connected components.
In fact, it is easy to construct trivial counter examples, like $\X = \Y = [0, 1] \cup [2, 3]$ with uniformly distributed $\mu = \nu$, where uniqueness of the Kantorovich potentials is well-known to fail for typical cost functions (see \cref{lem:ambiguous} in \cref{app:proofs}). 
At the same time, however, dual uniqueness properties for transportation problems on finite spaces have long been established via techniques from linear programming \parencite{klee1968facets,hung1986degeneracy}.
Even though these settings trivially involve disconnected supports, they do feature unique Kantorovich potentials whenever the measures $\mu$ and $\nu$ are \emph{non-degenerate}, meaning that all proper non-empty subsets $\X'\subset\X$ and $\Y'\subset\Y$ are assigned different masses $\mu(\X') \neq \nu(\Y')$.

The main contribution of our work is to make this classical finite-space observation accessible to the continuous world.
This is realized by lifting the uniqueness of dual solutions on connected components to uniqueness on the full space under suitable regularity guarantees.

\begin{theorem}{}{uniqueintro}
  Let $\X$ and $\Y$ be Polish, $\mu \in\PC(\X)$, $\nu \in \PC(\Y)$, and $c\colon
  \XY\to\RR_+$ continuous with $\OT_c(\mu, \nu) < \infty$.
  Let $\smash{(X_i)_{i\in I}}$ and $\smash{(Y_j)_{j\in J}}$ be the connected components of the supports of $\mu$ and $\nu$, where the index sets $I$ and $J$ are assumed to be finite.
  If
  \begin{enumerate}[topsep=1ex, parsep=0.5ex, font={\normalfont\color{orange!45!black}}, label=\theenumi)]
    \item $f^c$ is continuous on the support of $\nu$ for all Kantorovich potentials $f$,
    \item $\mu$ and $\nu$ are non-degenerate, meaning $\mu\big(\bigcup_{i\in \tilde{I}} X_i\big) \neq \nu\big(\bigcup_{j\in \tilde{J}} Y_j\big)$ for all proper non-empty subsets $\smash{\tilde{I}}\subset I$ and $\smash{\tilde{J}}\subset J$,
    \item there is an optimal $\pi\in\couple(\mu, \nu)$ with unique Kantorovich potentials for the optimal transport problem between $\mu_i = \mu|_{X_i}/\mu(X_i)$ and $\nu_i = \pi(X_i \times \cdot)/\mu(X_i)$ for all $i\in I$,
  \end{enumerate}
  then the Kantorovich potentials between $\mu$ and $\nu$ are uniquely determined.
\end{theorem}


At the expense of a more technical presentation, the first two conditions in \cref{thm:uniqueintro} can be weakened.
This is formalized in the more general \cref{thm:uniquedisconnected} in \cref{sec:uniquedisconnected}. Here, the index set $J$ is relaxed to be countable, and conditions that allow $I$ to be countable are provided as well (\cref{cor:uniquedisconnected}).
%
The proof of these statements takes advantage of the fact that non-degeneracy causes the Kantorovich potentials on the individual connected components of $\mu$ to be \enquote{linked} in a unique way (see \cref{sec:uniquenessproof} for proof details).
Independently of our work, a related linking idea was recently also used by \textcite{yang2023} to derive a more specialized result on uniqueness of dual optimal transport potentials in Euclidean spaces.

In order to apply \cref{thm:uniquedisconnected}, one requires criteria when Kantorovich potentials on connected domains (and their c-transforms) can be assumed to be continuous (condition 1) and unique (condition 3).
To systematically address the first of these issues, we introduce the notion of \emph{induced regularity} (\cref{def:regularity}) and harness it to document general continuity properties of Kantorovich potentials (\cref{lem:regularitycompact} and the remarks thereafter).
To tackle the second issue, we carefully re-examine
the established proof strategy for uniqueness in the connected setting under the lense of induced regularity.
By cautiosly dissecting which assumptions on $\mu$, $\nu$, and $c$ are actually necessary for the argumentation, we derive a novel uniqueness statement that generalizes many of the hitherto published results.
This statement, \cref{thm:uniqueconnected} in \cref{sec:uniqueconnected}, covers settings where $\X$ is a smooth manifold and $\Y$ is allowed to be a generic Polish space.
Particular scenarios where the assumptions of \cref{thm:uniqueconnected} can
easily be checked include settings where the space $\Y$ is compact
(\cref{cor:uniqueconnectedcompact}) or where the cost function $c$ satisfies sufficient regularity conditions outside of compact sets
(\cref{cor:uniqueconnectedlipschitz} and \ref{cor:uniqueconnectedboundary}).
We also derive novel results, which generalize and complement classical observations by \textcite{gangbo1996geometry}, for the pertinent case of sufficiently quickly growing cost functions in Euclidean (\cref{thm:euclideaninteriorregularity}) as well as geodesic spaces (\cref{thm:interiorregularity}).

Altogether, our findings highlight that the requirements for unique Kantorovich potentials, both in connected and disconnected settings, are often primarily of topological nature, i.e., concerning the shape of the support of $\mu$ or $\nu$.
In particular, we establish that dual uniqueness can be expected to hold much more frequently than previously anticipated.
For example, in the setting $\X = \Y = [0, 1] \cup [2, 3]$ mentioned earlier, combining \cref{thm:uniquedisconnected} and \cref{cor:uniqueconnectedcompact} establishes that Kantorovich potentials are indeed unique if $\mu$ and $\nu$ are supported on all of $\X$ and satisfy $\mu([0, 1]) \neq \nu([0, 1])$.
This holds for general differentiable costs without further assumptions on $\mu$ or $\nu$ such as the existence of a Lebesgue density.
In this sense, failure of uniqueness due to disconnected supports is typically found to be an exception caused by a specific degeneracy, and not the rule.

The article is structured as follows.
Relevant literature is surveyed in \cref{sec:related}.
Afterwards, the notion of $c$-concavity is introduced and the regularity of Kantorovich potentials is discussed in \cref{sec:potentials}.
Some technical results that cope with restrictions of the base spaces $\X$ and $\Y$ are emphasized as well.
\cref{sec:uniqueness} opens with a clarification of equivalent ways to define almost surely unique Kantorovich potentials.
Afterwards, uniqueness statements for probability measures with disconnected support are presented and discussed, including its consequences for the semi-discrete setting and countably discrete spaces (\cref{sec:uniquedisconnected}).
We then turn to uniqueness statements for measures with connected support on smooth manifolds (\cref{sec:uniqueconnected}).
\cref{sec:interior} contributes some findings on continuity properties of Kantorovich potentials for fast-growing cost functions, revealing that discontinuities are often confined to the boundary of the support.
Finally, \cref{sec:uniquenessproof} contains the proof of \cref{thm:uniquedisconnected} and \cref{app:proofs} documents auxiliary observations and arguments that have been omitted in the main text.

\paragraph{Notation.}
Throughout the manuscript, $\X$ and $\Y$ denote Polish spaces, i.e.,  completely metrizable and separable topological spaces.
For $A\subset\X$, we write $\interior(A)$ for its interior, $\closure(A)$ for its closure, and $\partial A = \closure(A)\setminus\interior(A)$ for its boundary.
The Cartesian projection from a product of spaces to a component $\X$ is denoted by $\pr_\X$.
Real-valued functions $f$ and $g$ on spaces $\X$ and $\Y$ can be lifted to $\XY$ via the operation $(x, y) \mapsto f(x) + g(y)$, which we denote by $f\oplus g$.
The set of Borel probability measures, or distributions, on a Polish space $\X$ are called $\PC(\X)$.
The support of a probability distribution $\mu\in\PC(\X)$, which is the smallest closed set $A\subset\X$ such that $\mu(A) = 1$, is denoted by $\supp\,\mu$.
We write $\mu\otimes\nu\in\PC(\XY)$ to denote the product of probability measures on Polish spaces $\X$ and $\Y$. Integration $\int\! f \dif\mu$ of a real-valued function $f$ on $X$ is abbreviated by juxtaposition $\mu f$.

If $M$ is a smooth manifold (without boundary), we call $f\colon M \to \RR$ locally Lipschitz if $f \circ \varphi^{-1}$ is locally Lipschitz for every chart $\varphi$ of an atlas of $M$.
Similarly, we call $f$ locally semiconcave if $f \circ \varphi^{-1}$ is locally semiconcave for every chart $\varphi$ of an atlas of $M$.
By this we mean that each point in $\range\,\varphi$ admits $\lambda > 0$ and a convex neighbourhood $V \subset \range\,\varphi$ such that 
\begin{equation}\label{eq:semiconcave}
  v \mapsto f\big(\varphi^{-1}(v)\big) - \lambda\|v\|^2
\end{equation}
is concave on $V$.
A family of functions $f_y\colon M\to\RR$ for $y\in\Y$ is called \emph{locally Lipschitz (or locally semiconcave) uniformly in $y$} if $f_y$ is locally Lipschitz (or locally semiconcave) with neighborhoods and constants that do not depend on $y$.
Similarly, the functions $f_y$ are \emph{locally Lipschitz locally uniformly in $y$} if the functions $f_y$ are locally Lipschitz uniformly in $y\in K$ for each compact set $K\subset\Y$.
We furthermore say that a Borel set $A$ has \emph{full Lebesgue measure in charts of $M$} if $\range\,\varphi\setminus\varphi\big(A \cap \domain\,\varphi\big)$ is a Lebesgue null set for each chart $\varphi$ of an atlas of $M$.

\section{Related work}
\label{sec:related}

While our efforts in this work focus on Kantorovich potentials, most of the foundational research on the optimal transport problem \eqref{eq:ot} has targeted properties of the primal solutions.
Significant advances, which provided sufficient conditions for optimal plans to be concentrated on the graph of a uniquely determined function (the \emph{optimal transport map}), have been achieved in Euclidean spaces \parencite{smith1987note,cuesta1989notes,brenier1991polar,gangbo1996geometry}, on manifolds \parencite{mccann2001polar,figalli2007existence,villani2008optimal,figalli2011local}, and more recently also in more general metric spaces \parencite{bertrand2008existence,gigli2012optimal,ambrosio2014slopes}.
Strong regularity properties of optimal transport maps under squared Euclidean costs have first been established by the seminal work of \citeauthor{caffarelli1990localization} (\citeyear{caffarelli1990localization},\citeyear{caffarelli1991some},\citeyear{caffarelli1992regularity}) for probability measures with bounded convex support (with recent extensions to unbounded settings by \textcite{Cordero2019Regularity}).
Further insights were obtained by \textcite{ma2005regularity} and \textcite{loeper2009regularity} for $\mathcal{C}^4$-costs that satisfy a certain differential inequality (the \emph{Ma-Trudinger-Wang condition}).
Later, \textcite{figalli2015partial} demonstrated regularity of optimal maps outside of \enquote{bad sets} of measure zero under more general conditions.
A related line of research is devoted to the analysis of optimal transport plans that are not necessarily induced by a transport map. Uniqueness results for the primal solution in this context were, for example, obtained by \textcite{ahmad2011optimal} or \textcite{mccann2016intrinsic}, and more recently by \textcite{moameni2020uniquely}.

Many of the techniques employed to characterize the primal solutions of \eqref{eq:otprimal} crucially depend on the duality theory \eqref{eq:otdual}.
In fact, the gradients of dual solutions are intimately related to optimal transport maps (see \cite[Chapter~10]{villani2008optimal}, for an in-depth treatment).
Still, dual solutions are commonly not studied as objects of interest in their own right, and certain properties, such as the uniqueness of Kantorovich potentials, have received considerably less attention when compared to their primal counterparts.
Recent developments, however, have emphasized the utility of dual uniqueness. 
For example, in the context of statistical optimal transport, uniqueness of Kantorovich potentials ensures a Gaussian limit distribution for the empirical optimal transport cost \parencite{sommerfeld2018inference, del2019central, tameling2019empirical, del2021centralSemidiscrete, del2021central, hundrieser2024unifying, hundrieser2024empirical}.
This is related to the observation that unique Kantorovich potentials cause the functional derivative of the optimal transport cost with respect to one of its argument measures (i.e., $\mu$ or $\nu$) to be linear.
In this case, the unique Kantorovich potential $f$ coincides with the subgradient of the functional $\mu \mapsto T_c(\mu, \nu)$ (\cite[Section~7.2]{santambrogio2015optimal}), which makes uniqueness a common requirement when working with optimal transport gradient flows in $\PC(\X)$ (see \cite{ambrosio2006} for a comprehensive treatment of gradient flows in probability spaces).
Furthermore, recent results on the convergence of entropically regularized optimal transport to its vanilla counterpart (as the regularization parameter tends to zero) utilize dual uniqueness as a critical assumption as well \parencite{altschuler2021asymptotics,bercu2021asymptotic, bernton2021entropic,nutz2021entropic, hundrieser2024limit}.

For continuous measures in Euclidean spaces, several sufficient criteria for unique Kantorovich potentials are available.
The involved arguments are well-known from the above mentioned literature on optimal transport maps, and depend on (i) sufficient local regularity of dual solutions $f$ (e.g., local Lipschitz continuity)  and (ii) exploiting that the gradient of $f$ is (where it exists) determined by the cost function and an (arbitrary) optimal transport plan.
Notable results that adopt this strategy include Proposition~7.18 in \textcite{santambrogio2015optimal}, which is applicable in compact settings, or Appendix~B of \textcite{bernton2021entropic} and Corollary~2.7 of \textcite{del2021central}, both relying on regularity properties of dual solutions derived by \textcite{gangbo1996geometry} for a certain family of strictly convex costs.
Meanwhile, Remark~10.30 in \textcite{villani2008optimal} sketches a general argument for uniqueness of Kantorovich potentials on Riemannian manifolds.

\section{Kantorovich Potentials}
\label{sec:potentials}

Let $c\colon\XY \to \RR_+$ be a non-negative cost function that compares elements of Polish spaces $\X$ and $\Y$.
As laid out comprehensively in \textcite{villani2008optimal} or \textcite{santambrogio2015optimal}, a central part of the duality theory of optimal transport is the notion of $c$-conjugacy.
For any $g\colon\Y\to\RRext$, its associated \emph{$c$-transform} is defined via
\begin{subequations}
\begin{equation}\label{eq:ctransformY}
  g^c\colon\X\to\RRext,
  \qquad
  g^c(x) = \inf_{y\in\Y} c(x, y) - g(y).
\end{equation}
Any function $f\colon\X\to\RRext$ that coincides with $g^c$ for some $g\colon\Y\to\RRext$ and that is not equal $-\infty$ everywhere is called \emph{$c$-concave on $\X$}.
The set of all functions that are $c$-concave on $\X$ is denoted by $\CCc$.
Since the roles of $f$ and $g$ can easily be exchanged in these definitions, we also write
\begin{equation}\label{eq:ctransformX}
  f^c\colon\Y\to\RRext,
  \qquad
  f^c(y) = \inf_{x\in\X} c(x, y) - f(x)
\end{equation}
\end{subequations}
for the $c$-transform of a function $f\colon\X\to\RRext$.
Any $g\colon\Y\to\RRext$ that originates from a $c$-transform and that is not equal $-\infty$ everywhere is called \emph{$c$-concave on $\Y$}.
Since $f = f^{cc}$ and $g = g^{cc}$ for any $c$-concave $f$ or $g$ (see \cite[Proposition~1.34]{santambrogio2015optimal}), the set $\CCcc$ of pointwise $c$-transformed elements of $\CCc$ equals the set of functions that are $c$-concave on $\Y$.
Under continuity of the cost function $c$, all functions in $\CCc$ and $\CCcc$ are upper-semicontinuous and thus Borel measurable.
Note that our notation accentuates the asymmetry in the operations \eqref{eq:ctransformY} and \eqref{eq:ctransformX} less explicitly than \textcite{santambrogio2015optimal}, who denotes \eqref{eq:ctransformY} as $\bar{c}$-transform, or \textcite{villani2008optimal}, who picks a different sign convention and contrasts $c$-concavity to $c$-convexity.

For any given $f\colon\X\to\RRext$, the $c$-transform $g = f^c$ designates the largest function that satisfies $f \oplus g \le c$.
The set of points in $\XY$ where equality holds is denoted as \emph{$c$-subdifferential of $f$}, and we write
\begin{equation}\label{eq:subdifferential}
  \partial_c f
  =
  \big\{(x,y)\in\XY\,\big|\, f(x) + f^c(y) = c(x, y)\big\}.
\end{equation}
This set is closed when $c$ is continuous and $f$ is upper-semicontinuous (so in particular when $f$ is $c$-concave).
If $f$ is a solution of the dual optimal transport problem \eqref{eq:otdual}, it is clear that any optimal transport plan has to be concentrated on $\partial_c f$.
The following statement, which is a special case of Theorem~5.10 (ii) in
\textcite{villani2008optimal}, establishes that (generalized) $c$-concave dual
solutions exist under mild conditions.

\begin{theorem*}{Existence of optimal solutions}
  Let $\X$ and $\Y$ be Polish and $c\colon\XY\to\RR_+$ continuous.
  For any $\mu\in\PC(\X)$ and $\nu\in\PC(\Y)$ with $\OT_c(\mu, \nu) < \infty$, there exists an optimal transport plan $\pi\in\couple(\mu, \nu)$ and a $c$-concave function $f\in\CCc$ such that
  \begin{equation}\label{eq:existence}
    \OT_c(\mu, \nu) = \pi c = \pi \big(f \oplus f^c\big).
  \end{equation}
\end{theorem*}

We emphasize that the function $f$ in this statement does not have to be $\mu$-integrable, nor does $f^c$ have to be $\nu$-integrable.
Ensuring integrability requires further conditions \parencite[Remark~5.14]{villani2008optimal}, for instance $(\mu\otimes\nu)\,c < \infty$.
Only then can $f$ be viewed as a dual optimizer of \eqref{eq:otdual} in the strict sense
\begin{equation*}
  T_c(\mu, \nu) = \pi c = \mu f + \nu f^c.
\end{equation*}
For our ends, however, the more general solutions provided by \eqref{eq:existence} are sufficient.
We call these solutions \emph{(generalized) Kantorovich potentials}, and we write $f\in\CCc(\mu, \nu)\subset\CCc$ or $f^c\in\CCcc(\mu, \nu)\subset\CCcc$ to emphasize their dependence on $\mu$ and $\nu$.
We stress that any $f\in\CCc(\mu, \nu)$ satisfies \eqref{eq:existence} for \emph{all} optimal transport plans $\pi$, so $f$ does not favor a particular primal solution \parencite[Lemma~1.1]{beiglbock2011duality}.

Note that the existence of solutions as well as duality statements for optimal transportation problems have also been established for non-continuous cost functions \parencite{villani2008optimal,beiglbock2011duality} or more general spaces \parencite{ruschendorf2007monge}.
Two major advantages of working with continuous costs are the closedness of the $c$-subdifferential $\partial_c f$ for any $f\in\CCc(\mu, \nu)$ and the (related) upper-semicontinuity of $c$-conjugate functions.
The former implies
\begin{equation*}
  \supp\,\pi\subset\partial_c f
\end{equation*}
for any optimal transport plan $\pi$, a property which we will often resort to, while the latter is needed for continuity results and permits sidestepping measurability issues.

\subsection{Regularity}
Due to their nature as $c$-concave functions, Kantorovich potentials inherit certain regularity properties from the cost function $c$.
For example, if $c$ is semiconcave in its first argument, then $f\in\CCc$ is also semiconcave as an infimum over semiconcave functions.
Similarly, if the family $\{c(\cdot, y)\,|\,y\in\Y\}$ of partially evaluated costs is (locally) equicontinuous, then $f$ shares the respective (local) modulus of continuity (\cite[Section~1.2]{santambrogio2015optimal}).
Imposing conditions of this form is hence a convenient way to guarantee continuity of $c$-concave functions, which are in general only upper-semicontinuous (for continuous costs).
In the following, we introduce tools that give us a more fine-grained control over the continuity of Kantorovich potentials.
We begin with continuity along sequences in $\partial_c f$.

\begin{lemma}{}{regularity}
  Let $X$ and $Y$ be Polish, $c\colon\XY\to\RR_+$ continuous, and $f\in\CCc$.
  If $(x_n, y_n)_{n\in\NN}$ is a sequence in $\partial_cf$ that converges to $(x, y)\in\partial_c f$, then $f(x_n) \to f(x)$ and $f^c(y_n) \to f^c(y)$ as $n\to\infty$.
\end{lemma}

\begin{proof}
  Both $f$ and $f^c$ are upper-semicontinuous. Since $(x_n, y_n)$ and $(x, y)$ are elements in $\partial_c f$,
  \begin{equation*}
    f^c(y)
    \ge
    \limsup_{n\to\infty} f^c(y_n) 
    \ge
    c(x, y) - \limsup_{n\to\infty} f(x_n)
    \ge
    c(x, y) - f(x) = f^c(y).
  \end{equation*}
  Therefore, $\limsup_{n\to\infty} f(x_n) = f(x)$ and $\limsup_{n\to\infty} f^c(y_n) = f^c(y)$ has to hold.
\end{proof}

\begin{figure}
  \centering\small
  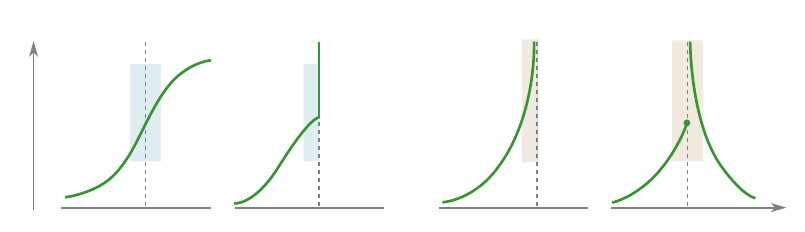
  \caption{Transport plans and induced regularity.
  The two plans on the left induce regularity at $x\in\supp(\mu)$ (dashed line), since a relatively open neighborhood $U\subset\supp(\mu)$ and a compact set $K\subset\Y$ can be found such that condition \eqref{eq:partiallycompact} is satisfied.
  The two plans on the right fail to induce regularity. Note that $x\in\pr_\X(\supp\,\pi)$ is possible even if regularity is not induced at the point $x$ (rightmost sketch).
  }
  \label{fig:partiallycompact}
\end{figure}

We next show that Kantorovich potentials can only be discontinuous at points that are \enquote{sent to infinity} by all optimal transport plans.
To put this more precisely, we introduce the notion of \emph{induced regularity} (see \cref{fig:partiallycompact}).

\begin{definition}{Induced regularity}{regularity}
  Given a relatively open set $U\subset\supp\,\mu$, we say that a transport plan $\pi\in\couple(\mu, \nu)$ \emph{induces regularity on} $U$ if there exists a compact set $K\subset\supp\,\nu$ such that
  \begin{equation}\label{eq:partiallycompact}
    \pr_\X(\supp\,\pi) \cap U
    =
    \pr_\X\big(\supp\,\pi \cap (U\times K)\big),
  \end{equation}
  where $\pr_\X$ denotes the coordinate projection onto $\X$. We also say that $\pi$ \emph{induces regularity at} $x\in\supp\,\mu$, if there is a relatively open neighborhood $U\subset\supp\,\mu$ of $x$ such that $\pi$ induces regularity on $U$.
  The set of points where $\pi$ fails to induce regularity is denoted as
  \begin{equation}\label{eq:sigma}
    \Sigma_\pi
    =
    \big\{x\in\supp\,\mu\,\big|\,\pi~\text{does not induce regularity at}~x\big\}.
  \end{equation}
\end{definition}

Similar definitions with reversed roles of $\mu$ and $\nu$ can be deployed for subsets or points in $\supp\,\nu$ as well, to which all of the following statements may be adjusted straightforwardly.

\begin{lemma}{}{regularitycompact}
  Let $X$ and $Y$ be Polish, $\mu\in\PC(\X)$, $\nu\in\PC(\Y)$, and $c\colon\XY\to\RR_+$ continuous with $\ot_c(\mu, \nu) < \infty$. Let $\pi\in\couple(\mu, \nu)$ be optimal and $f\in\CCc(\mu, \nu)$.
  If $\pi$ induces regularity on $U\subset\supp\,\mu$ with respect to the compact set $K\subset\supp\,\nu$, then $U\subset\pr_\X(\supp\,\pi)$ and 
  \begin{equation}\label{eq:regularitycompact}
    f(x) = \inf_{y\in K} c(x, y) - f^c(y)
  \end{equation}
  for all $x\in U$. In particular, $f|_{\supp\,\mu}$ is continuous on $U$.
\end{lemma}

\begin{proof}
  Let $\pi$ induce regularity on the relatively open set $U\subset\supp\,\mu$ with respect to the compact set $K\subset\supp\,\nu$ as defined in \eqref{eq:partiallycompact}. 
  Restricted to the domain $\X\times K$, the projection $\pr_\X$ is a closed map.
  Condition \eqref{eq:partiallycompact} thus establishes that $A = \pr_\X(\supp\,\pi)\cap U = \pr_\X\big(\supp\,\pi \cap (X\times K)\big) \cap U$ is relatively closed in $U$.
  Since $\pr_\X(\supp\,\pi)$ is dense in $\supp\,\mu$, any point in $U$ is a limit point of $A$.
  Therefore, $U = A \subset \pr_\X(\supp\,\pi)$.
  Furthermore, each $x\in U$ admits a partner $y\in K$ such that $(x, y)\in\supp\,\pi\subset\partial_c f$.
  This establishes \eqref{eq:regularitycompact}, since
  \begin{equation*}
    f(x)
    =
    c(x, y) - f^c(y)
    =
    \inf_{y'\in K} c(x, y') - f^c(y').
  \end{equation*}
  To show continuity of $f|_{\supp\,\mu}$ on $U$, it is enough to show that each sequence $(x_n)_{n\in\NN}\subset\supp\,\mu$ converging to $x\in U$ has a subsequence attaining the limit $f(x)$.
  Since $U$ is (relatively) open, we can assume $(x_n)_{n}\subset U$.
  Thus, each $x_n$ admits $y_n\in K$ such that $(x_n, y_n)\in \supp\,\pi$.
  After taking a suitable subsequence, we may assume that $y_n \to y\in K$ due to the compactness of $K$.
  Thus, $(x_n, y_n) \to (x, y)$ as $n\to\infty$.
  Since $\supp\,\pi \subset \partial_c f$ is closed in $\XY$, it contains $(x,y)$, and so \cref{lem:regularity} can be applied to establish $\lim_{n\to\infty}f(x_n) = f(x)$.
\end{proof}

\begin{remark}{Projections and measurability}{measurability}
  We commonly formulate our results in terms of the projected sets $\pr_\X(\supp\,\pi)\subset\supp\,\mu$ and $\pr_\Y(\supp\,\pi)\subset\supp\,\nu$, where $\pi\in\couple(\mu, \nu)$ denotes an (arbitrary) optimal transport plan.
  Note that these inclusions are possibly strict, in which case \cref{lem:regularitycompact} establishes the relation
  \begin{equation*}
    \supp\,\mu \setminus \pr_\X(\supp\,\pi)
    \subset
    \Sigma_\pi,
  \end{equation*}
  and vice versa for $\nu$.
  Projected sets of the form $\pr_\X(\supp\,\pi)$ or $\pr_\Y(\supp\,\pi)$ are not Borel measurable in general, but they are universally measurable and thus contain Borel subsets of full $\mu$- or $\nu$-measure (see \cref{lem:measurability} in \cref{app:proofs} for references).
  We prefer the explicit formulation via these sets (instead of writing \enquote{almost surely}) to emphasize that the domain of a given property does not depend on the choice of a specific Kantorovich potential.
\end{remark}

Some consequences of \cref{lem:regularitycompact} deserve to be highlighted.
First, the Kantorovich potentials $\CCc(\mu, \nu)$ are always continuous if the space $Y$ is compact.
If $Y$ is not compact, the intuition fostered by \cref{lem:regularitycompact} is that discontinuities can only occur at points from whose immediate vicinity some mass is sent towards infinity, in the sense that this mass leaves any compactum in $\Y$.
Relation \eqref{eq:regularitycompact} is particularly useful for transferring properties of the cost function to Kantorovich potentials, like a modulus of continuity of $c(\cdot, y)$ that holds only \emph{locally} in $y$.
This observation becomes crucial for \cref{thm:uniqueconnected} in \cref{sec:uniqueness}.

Examples where induced regularity fails at some points can easily be found, and include settings where $\mu$ is compactly supported but $\nu$ is not.
At the offending points, $f$ can turn out to be both continuous or discontinuous, depending on the regularity of the cost function as well as the specific behavior of $\mu$ and $\nu$ (this can already be observed for $\X = \Y = \RR$). 
The following example anticipates that points of discontinuity are often restricted to the boundary of the support, a phenomenon that we more closely investigate in \cref{sec:interior}.

\begin{example}{Continuity in geodesic spaces}{continuitygeodesic}
  Let $\X = \Y$ for a locally compact complete geodesic space $(\X, d)$ and consider a cost function of the form $c(x, y) = h(d(x, y))$ with convex and locally Lipschitz $h\colon\RR_+\to\RR_+$.
  Then every Kantorovich potential $f\in\CCc(\mu, \nu)$ is continuous on the interior of the support of $\mu$ and discontinuities at the boundary are only possible if $h'(a)\to\infty$ as $a\to\infty$.
  Proofs and more details are provided in \cref{sec:interior}.
\end{example}

We stress that regularity properties of Kantorovich potentials that go beyond mere continuity, like their degree of differentiability, are extensively studied in the literature on the optimal transport map (see \cite[Chapter~10]{villani2008optimal}, for a detailed exposition).
To mention a common argument in this context, one possible way to enforce twofold differentiability of each $f\in\CCc$ (Lebesgue-almost everywhere) is to work with semiconcave cost functions.

\begin{example}{Continuity under semiconcavity}{continuityconcave}
  Let $X = \RR^d$ for some $d\in\NN$ with the Euclidean norm $\|\cdot\|$ and assume that the function class $\big\{c(\cdot, y)\,\big|\,y\in\Y\big\}$ has uniformly bounded second derivatives.
  Then there exists $\lambda > 0$ such that $x\mapsto c(x, y) - \lambda \|x\|^2$ is concave for each $y\in\Y$, which implies concavity of $x\mapsto f(x) - \lambda \|x\|^2$ for $f\in\CCc(\mu, \nu)$ as well.
  In particular, $f$ is continuous and Lebesgue-almost everywhere twice differentiable in the domain $\Omega = \interior(\{x\,|\,f(x) > -\infty\})\subset\RR^d$, which is convex and contains the interior of $\supp\,\mu$.
  On the boundary of $\Omega$, the potential may assume finite values or $-\infty$.
\end{example}

\subsection{Restrictions}
A valuable trait of optimal transport theory is that both primal and dual solutions behave consistently when the base spaces $\X$ and $\Y$ are restricted to subspaces.
General results in this direction can be found in \cite[Theorem~4.6 and Theorem~5.19]{villani2008optimal}.
In the following, we stress some selected statements that complement our assertions on continuity and uniqueness of Kantorovich potentials.
We begin with the technical observation that restrictions of Kantorovich potentials of the form $f|_{\supp\,\mu}$, as they appear in \cref{lem:regularitycompact}, can (almost) be understood as Kantorovich potentials of a suitably restricted problem.
The proof is delegated to \cref{app:proofs}.

\begin{lemma}{Restriction to sets of full mass}{fullrestriction}
  Let $\X$ and $\Y$ be Polish and $c\colon\XY\to\RR_+$ continuous.
  Suppose $\mu\in\PC(\X)$ and $\nu\in\PC(\Y)$ such that $\OT_c(\mu, \nu) < \infty$.
  Let $\pi\in\couple(\mu, \nu)$ be an optimal plan and let $\cres$ denote the restriction of $c$ to the set $\Xres\times\Yres$, where $\Xres\subset \X$ and $\Yres\subset\Y$ are Borel and Polish subspaces with $\mu(\Xres) = \nu(\Yres) = 1$.
  Let $\tilde\Gamma = \supp\,\pi\cap(\XYres)$.
  \begin{description}[topsep=1.25ex, parsep=0.0ex, font={\normalfont\color{green!30!black}}]
    \item[(Restrict)~] Every $f\in \CCc(\mu, \nu)$ admits $\fres\in
      \CC_\cres(\mu, \nu)$ that agrees with $f$ on $\pr_\X(\tilde\osupp)$.
    \item[(Extend)~~] Every $\fres \in \CC_\cres(\mu, \nu)$ admits $f\in
      \CCc(\mu, \nu)$ that agrees with $\fres$ on $\pr_\X(\tilde\osupp)$.
  \end{description}
  In both cases, the conjugates $f^c$ and $\fres^\cres$ agree on $\pr_\Y(\tilde\osupp)$.
\end{lemma}

\vspace{-1ex}

\begin{remark}{Ambiguity of extensions}{ambiguous}
  Depending on the setting, there can be distinct ways of extending Kantorovich potentials from the support to the whole space.
  For example, if $\X = \Y$ are equal and $c$ is a metric, then the $c$-concave functions are exactly the $1$-Lipschitz functions with respect to $c$, and it is easy to see that $f^c = -f$ holds for any $f\in\CCc$.
  In this situation, ambiguous extensions are common if $\supp\,\mu \cup\supp\,\nu$ does not cover the whole space $\X$.
\end{remark}

We next address the behavior of Kantorovich potentials when transportation is restricted to a part of $\X$ that does not necessarily occupy full $\mu$-mass.
\begin{definition}{Restricted problem}{restricted}
  Let $\pi\in\couple(\mu, \nu)$ be an optimal plan and suppose that $\mu(\Xres) > 0$ for some closed subset $\Xres\subset\X$.
  We denote the optimal transport problem between the probability measures $\mu_\Xres = \mu|_\Xres / \mu(\Xres)$ and $\nu_\Xres = \pi(\Xres \times \cdot)/\mu(\Xres)$ under the cost function $c_\Xres = c|_{\Xres\times\Y}$ as the \emph{$\Xres$-restricted problem (with respect to $\pi$)}.
\end{definition}
As an application of \cite[Theorem~5.19]{villani2008optimal}, we note that the restriction of Kantorovich potentials in the original problem yields Kantorovich potentials in the restricted problem.

\begin{lemma}{Restriction to sets of partial mass}{restriction}
  Let $\X$ and $\Y$ be Polish and $c\colon\XY\to\RR_+$ continuous.
  Suppose $\mu\in\PC(\X)$ and $\nu\in\PC(\Y)$ such that $\ot_c(\mu, \nu) < \infty$ and let $\Xres\subset\X$ be closed with $\mu(\Xres) > 0$.
  Let an optimal plan $\pi$ be given.
  Then any $f\in \CCc(\mu, \nu)$ admits $\fres \in \CC_{c_\Xres}(\mu_\Xres, \nu_\Xres)$ that agrees with $f$ on $\pr_\X(\supp\,\pi)\cap \Xres$.
\end{lemma}

As a consequence of \cref{lem:restriction}, it is always possible to decompose Kantorovich potentials defined on disconnected spaces in a natural way.
Indeed, if $\supp\,\mu = \bigcup_{i\in I} \X_i$ is a countable partitioning into connected components (which are always closed) with $\mu(\X_i) > 0$ for each $i\in I$, then any $f\in\CCc(\mu,\nu)$ admits restricted potentials $f_i\in\CC_{c_{\X_i}}(\mu_{\X_i}, \nu_{\X_i})$ such that
\begin{equation}\label{eq:potentialdecomposition}
  f = \sum_{i\in I} \ind_{\X_i}\cdot f_i
  \qquad
  \text{on $\pr_\X(\supp\,\pi)$}
\end{equation}
for any optimal $\pi\in\couple(\mu, \nu)$.
This simple but crucial observation lies at the heart of the uniqueness result for probability measures with disconnected support discussed in the next section.

\section{Uniqueness}
\label{sec:uniqueness}

In a strict sense, Kantorovich potentials are never unique.
Indeed, it is easy to see that $f\in\CCc(\mu, \nu)$ implies $f + a\in\CCc(\mu, \nu)$ for any $a\in\RR$.
Therefore, statements about uniqueness are generally only reasonable up to constant shifts.
Besides this ambiguity, it is often too restrictive to require uniqueness to hold outside of the supports of the involved measures (see \cref{rem:ambiguous} on ambiguous extensions).
We will therefore focus on the notion of almost surely unique Kantorovich potentials up to constant shifts.

\begin{definition}{Uniqueness}{}
  We refer to \emph{unique Kantorovich potentials} if $f_1 - f_2$ is $\mu$-almost surely constant for all $f_1, f_2\in\CCc(\mu, \nu)$. 
\end{definition}

Due to the regularizing nature of the $c$-transform, this kind of uniqueness of $\CCc(\mu, \nu)$ is actually equivalent to $\nu$-almost sure uniqueness of $\CCcc(\mu, \nu)$ up to constants.

\begin{lemma}{}{uniqueness}
  Let $X$ and $Y$ be Polish, $\mu\in\PC(\X)$, $\nu\in\PC(\Y)$, and $c\colon\XY\to\RR_+$ continuous such that $\OT_c(\mu, \nu) < \infty$.
  For any optimal transport plan $\pi$ and any $f_1, f_2\in\CCc(\mu, \nu)$
  \begin{enumerate}[topsep=1ex, parsep=0.5ex, font={\normalfont\color{green!30!black}},label=\theenumi)]
    \item $f_1 = f_2$ $\mu$-almost surely if and only if $f_1 = f_2$ on $\pr_\X(\supp\,\pi)$,
    \item $f_1^c = f_2^c$ $\nu$-almost surely if and only if $f_1^c = f_2^c$ on $\pr_\Y(\supp\,\pi)$,
    \item $f_1 = f_2$ on $\pr_\X(\supp\,\pi)$ if and only if $f_1^c = f_2^c$ on $\pr_\Y(\supp\,\pi)$.
  \end{enumerate}
\end{lemma}

\vspace{-1ex}

\begin{proof}
  We begin with the first assertion and assume that $f_1 = f_2$ holds on a Borel set $A\subset \X$ with $\mu(A) = 1$.
  The set $B = \supp\,\pi \cap (A\times \Y)$ is dense in $\supp\,\pi$, so that there is a convergent sequence in $B$ to any $(x, y)\in\supp\,\pi$.
  \cref{lem:regularity} then asserts  $f_1(x) = f_2(x)$ for all $x\in\pr_\X(\supp\,\pi)$.
  Conversely, \cref{lem:measurability} in \cref{app:proofs} shows that $\pr_\X(\supp\,\pi)$ contains a Borel set of full $\mu$-measure.
  Assertion~2 follows similarly.
  To show assertion~3, it is sufficient to observe $c(x, y) = f_1(x) + f_1^c(y) = f_2(x) + f_2^c(y)$ and thus $f_1(x) - f_2(x) = f_2^c(y) - f_1^c(y)$ for any $(x, y)\in\supp\,\pi$.
\end{proof}

\subsection{Disconnected support}
\label{sec:uniquedisconnected}

Our contributions regarding the uniqueness of Kantorovich potentials for measures with disconnected support are inspired by well-known results from the theory of finite linear programming.
In a nutshell, we show that unique Kantorovich potentials on the connected components of the support are sufficient to imply the uniqueness on the whole support, as long as we have continuous $c$-transformed potentials and so-called non-degenerate optimal plans.

\begin{definition}{Degenerate plans}{degenerate}
Consider countable disjoint decompositions
\begin{equation}\label{eq:components}
  \supp\,\mu = \bigcup_{i\in I} X_i
  \qquad\text{and}\qquad
  \supp\,\nu = \bigcup_{j\in J} Y_j
\end{equation}
of the supports of $\mu$ and $\nu$ into their connected components.
Then a transport plan $\pi\in\couple(\mu, \nu)$ is called \emph{degenerate} if there exist subsets $I'\subset I$ and $J' \subset J$ such that
\begin{equation}\label{eq:degenerate}
  0
  <
  \sum_{i\in I'}\mu(\X_i)
  =
  \sum_{i\in I'}\sum_{j\in J'} \pi(\X_i\times \Y_j)
  =
  \sum_{j\in J'}\nu(\Y_j)
  <
  1.
\end{equation}
\end{definition}

Equation~\eqref{eq:degenerate} suggests the following sufficient non-degeneracy criterion, which can easily be checked on the basis of $\mu$ and $\nu$ alone.

\begin{lemma}{}{degeneracy}
  If all nonempty proper $I' \subset I$ and $J' \subset J$ satisfy $\sum_{i\in I'}\mu(\X_i) \neq \sum_{j\in J'}\nu(\Y_j)$, then no transport plan $\pi\in\couple(\mu, \nu)$ is degenerate.
\end{lemma}

Under suitable conditions, non-degenerate optimal transport plans make it possible to uniquely link together Kantorovich potentials of the $X_i$-restricted transport problems (recall \cref{def:restricted}) to assert uniqueness of Kantorovich potentials on the full support.
For the following result, note that $\mu(\X_i) > 0$ for all $i\in I$ if $I$ is finite, since each $\X_i\subset\supp\,\mu$ is relatively open in this case.

\begin{theorem}{Uniqueness, disconnected support}{uniquedisconnected}
  Let $\X$ and $\Y$ be Polish, $\mu \in\PC(\X)$, $\nu \in \PC(\Y)$, and $c\colon \XY\to\RR_+$ continuous with $\OT_c(\mu, \nu) < \infty$.
  Assume decomposition~\eqref{eq:components} for $I$ finite and $J$ (at most) countable and assume that for all $f^c\in\CCcc(\mu, \nu)$ either
  \begin{enumerate}[topsep=1ex, parsep=0.5ex, font={\normalfont\color{orange!45!black}}, label=\theenumi)]
    \item $f^c|_{\supp\,\nu}$ is continuous, or
    \item $f^c|_{\Y_j}$ is continuous and $\supp\,\nu|_{Y_j}$ is connected for
      all $j\in J$ with $\nu(Y_j) > 0$.
  \end{enumerate}
  If there exists a non-degenerate optimal transport plan $\pi\in\couple(\mu, \nu)$ with respect to which the $\X_i$-restricted Kantorovich potentials $\CC_{c_{\X_i}}\!(\mu_{\X_i}, \nu_{\X_i})$ are unique for all $i\in I$, the Kantorovich potentials $\CCc(\mu, \nu)$ are also unique.
\end{theorem}

\vspace{-1ex}

\begin{example}{Semi-discrete optimal transport}{semi-discrete}
  Let $\X$ be a finite set with $n\in\NN$ elements and $\Y$ be a general Polish space.
  Questions concerning dual uniqueness in this setting, which is referred to as \emph{semi-discrete optimal transport}, have recently been raised by \textcite{altschuler2021asymptotics} and \textcite{bercu2021asymptotic}.
  \cref{thm:uniquedisconnected} provides simple and general answers in this context, since the uniqueness of the $\X_i$-restricted Kantorovich potentials and the continuity of all $f^c\in\CCcc$ turn out to be trivial (for continuous $c$).
  For example, if $\nu\in \PC(\Y)$ has connected support, Kantorovich potentials are \emph{always} unique in the above sense for \emph{any} $\mu\in\PC(\X)$.
  If the support of $\nu$ is disconnected with (at most) countably many components $\Y_j$, uniqueness holds if the measures $\mu$ and $\nu$ are non-degenerate in the sense of \cref{lem:degeneracy}.
  In particular, when $\mu$ is the uniform distribution on $\X$, Kantorovich potentials are guaranteed to be unique unless $\nu$ assigns a multiple of mass $1/n$ to individual connected components of $\supp\,\nu$.
\end{example}

\begin{figure}
  \begin{center}\small
    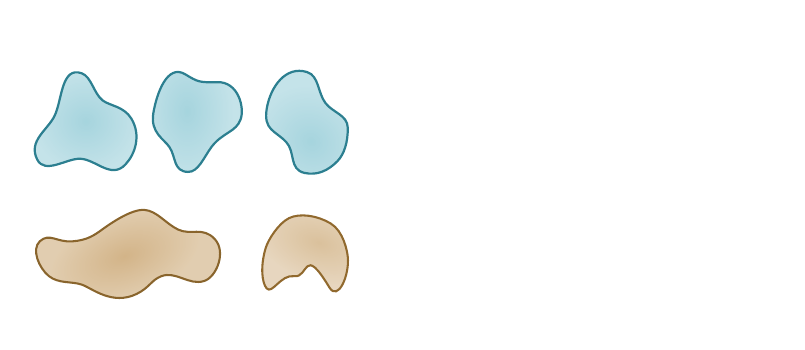
  \end{center}
  \caption{%
    Uniqueness of Kantorovich potentials in disconnected spaces.
    In sketch \textbf{(a)}, transport between $\X_1 \cup \X_2$ and $\Y_1$ is decoupled from transport between $\X_3$ and $\Y_2$.
    As the proof of \Cref{thm:uniquedisconnected} shows, the existence of a contact point $y_{1}$ links the (restricted) Kantorovich potentials $f_1\colon \X_1\to\RR$ and $f_2\colon\X_2\to\RR$.
    However, since $f_3\colon \X_3\to\RR$ is linked to neither $f_1$ nor $f_2$, uniqueness of the full Kantorovich potential $f\colon\X\to\RR$ is not guaranteed.
    In sketch \textbf{(b)}, $f_1$ is linked to $f_2$ via $y_{1}$ and $f_2$ is linked to $f_3$ via $y_{2}$.
    Therefore, uniqueness of $f$ follows from uniqueness of the restricted Kantorovich potentials if $f^c$ is continuous at $y_1$ and $y_2$.
  }
  \label{fig:uniquedisconnected}

\end{figure}

A sketch that assists in the interpretation of \cref{thm:uniquedisconnected} is provided by Figure~\ref{fig:uniquedisconnected}.
At the heart of the proof lies observation~\eqref{eq:potentialdecomposition}, which ensures that each Kantorovich potential in $\CCc(\mu, \nu)$ assumes the form 
\begin{equation*}
  f_a = \sum_{i\in I} \ind_{\X_i} \cdot (f_i + a_i)
  \qquad\text{on $\pr_\X(\supp\,\pi)$}
\end{equation*}
for some $a\in\RR^{|I|}$, where representatives $f_i\in\CC_{c_{\X_i}}(\mu_{\X_i}, \nu_{\X_i})$ of the $\X_i$-restricted problems have been fixed.
Not each choice of $a$ leads to a viable optimal solution $f_a\in S_c(\mu, \nu)$, however.
Due to the continuity of $f_a^c$, one can show that a value $a_{i_1}$ uniquely determines $a_{i_2}$ for $i_1, i_2\in I$ if the masses transported from $\X_{i_1}$ and $\X_{i_2}$ to $\Y$ touch one another, meaning that their topological closures have a common \emph{contact point} (see \cref{sec:uniquenessproof} for formal definitions).
The non-degeneracy of $\pi$ then ensures the existence of suitable contact points such that fixing $a_{i_0}$ for an arbitrary $i_0\in I$ actually determines the whole vector $a$, which in consequence implies dual uniqueness.
The full proof is documented in \cref{sec:uniquenessproof}, while the following paragraphs discuss the assumptions in \cref{thm:uniquedisconnected}.

\paragraph{Degeneracy.}

The uniqueness of Kantorovich potentials can break down if the condition of non-degeneracy of $\pi$ is not satisfied.
For a simple family of examples, consider non-negative continuous and symmetric costs with $c(x, x) = 0$ for $x\in\X = \Y$ in a setting where $\mu = \nu$ with $\supp(\mu) = \X_1 \cup \X_2$.
If the components $\X_1$ and $\X_2$ are strictly cost separated,
\begin{equation}\label{eq:strictlyseparatedcosts}
    \Delta = \inf_{x_1\in\X_1, x_2\in\X_2} c(x_1, x_2) > 0,
\end{equation}
each optimal plan $\pi$ satisfies $\pi(\X_1\times\X_2) = \pi(\X_2\times\X_1) = 0$.
In particular, no mass is transported between $\X_1$ and $\X_2$.
In this situation, any pair $a,b\in \RR$ with $|a-b|\leq \Delta$ defines a Kantorovich potential $f_{a,b}$ via $f_{a,b} = a$ on $\X_1$ and $f_{a,b} = b$ on $\X_2$.
A proof of this observation is provided in \cref{app:proofs}, \cref{lem:ambiguous}.

\paragraph{Continuity.}

Even in the presence of non-degenerate optimal transport plans, the uniqueness of Kantorovich potentials can in principle still break down due to discontinuities of the cost function or the potentials.
For instance, the construction of non-unique potentials $f_{a,b}$ in \cref{lem:ambiguous} under a $\Delta$-separation between components of a disconnected support can easily be carried over to cost functions that exhibit a jump by $\Delta$ between two disjoint subsets of $\supp\,\mu$, even if the support is connected.

For continuous costs, we discussed in \cref{sec:potentials} that the $c$-transformed Kantorovich potentials $f^c\in\CCcc(\mu, \nu)$ are always continuous when the family $\{c(x, \cdot)\,|\,x\in\X\}$ of partially evaluated costs is (locally) equicontinuous, or when the space $\X$ is compact (\cref{lem:regularitycompact}).
We also note that conditions~1 and 2 in \cref{thm:uniquedisconnected} can actually be relaxed, and it would in both cases suffice to require continuity only at the finite number of contact points (implicitly) constructed in the proof.
According to \cref{lem:regularitycompact}, this is for example guaranteed if $\pi$ induces regularity at each contact point $y\in\Y$.
In settings with $\nu(\partial\,\supp\,\nu) = 0$, we can sometimes even drop the additional continuity assumption altogether: if the functions $f^c$ are known to be continuous in the interior of $\supp\,\nu$, which is true for a wide range of superlinear costs (see \cref{sec:interior}, which in particular covers \cref{ex:continuitygeodesic}) or costs with uniformly bounded second derivatives (see \cref{ex:continuityconcave}), then we can apply \cref{lem:fullrestriction} to transition to the restricted problem with $\Xres = \X$ and $\Yres = \interior(\supp\,\nu)\subset \Y$.
This reformulation, where one only has to heed possible changes in decomposition \eqref{eq:components} when replacing $\Y$ by $\Yres$ (which may affect the degeneracy of optimal transport plans), makes sure that suitable contact points $y$ can always be  found in the interior of $\supp\,\nu$.

\paragraph{Countable index sets.}

Due to topological complications in the proof, the decomposition of $\supp\,\mu$ in \cref{thm:uniquedisconnected} is restricted to a finite index set $I$.
Indeed, if mass of an infinite number of components $\X_i$ is transported to a single component $\Y_j$, our proof technique cannot be used to establish the existence of suitable contact points.
Two settings where this issue can easily be reconciled is when 1) all sets $\Y_j$ receive mass from finitely many $\X_i$ only, or when 2) each $\Y_i$ is a single point.

\begin{corollary}{}{uniquedisconnected}
  Under either of the following additional assumptions, \cref{thm:uniquedisconnected} remains valid for countable index sets $I$, where uniqueness of the $\X_i$-restricted Kantorovich potentials is only required for $i\in I$ with $\mu(\X_i) > 0$.
  \begin{enumerate}[topsep=1ex, parsep=0.5ex, font={\normalfont\color{orange!45!black}}, label=\theenumi)]
    \item Condition~2 in \cref{thm:uniquedisconnected} holds and $\big|\{i\in I\,|\,\pi(X_i\times Y_j) > 0\}\big| < \infty$ for all
      $j\in J$.
    \item $Y_j$ consists of a single point for each $j\in J$ with $\nu(Y_j)
      > 0$ (then condition~2 in \cref{thm:uniquedisconnected} is always satisfied).
  \end{enumerate}
\end{corollary}

In the special case of statement 2 of \cref{cor:uniquedisconnected} where all components $\X_i$ are also single points, we conclude that non-degeneracy of the vectors $(\mu(X_i))_{i\in I}$ and $(\nu(Y_j))_{j\in J}$ as in \cref{lem:degeneracy} is already sufficient to imply uniqueness of Kantorovich potentials without further assumptions.
This criterion has long been established for finite transportation problems via the theory of finite linear programming \parencite{klee1968facets,hung1986degeneracy}, but we are not aware of a comparable result that covers probability measures with countable support.
In particular, we note that \cref{cor:uniquedisconnected} yields sufficient conditions for Gaussian distributional limits in the countable discrete settings studied by \textcite{tameling2019empirical}.

\subsection{Connected support}
\label{sec:uniqueconnected}

One piece that is still missing to fully utilize \cref{thm:uniquedisconnected} is a set of criteria for the uniqueness of Kantorovich potentials on the individual connected components of the support.
In Euclidean settings, results in this regard are readily available, see for example Proposition~7.18 in \textcite{santambrogio2015optimal} (compactly supported measures) or more recently Appendix~B of \textcite{bernton2021entropic} and Corollary~2.7 of \textcite{del2021central} (possibly non-compactly supported measures).
The necessary techniques for these statements have long been established \parencite{brenier1991polar, gangbo1996geometry} and have been extended to the more general setting of manifolds \parencite{mccann2001polar,villani2008optimal,fathi2010optimal, figalli2011local}.
In the following, we briefly revisit the underlying arguments and spell out a general uniqueness result together with some of its consequences for probability measures with connected support.
To our knowledge, no formal statement with comparable scope has yet been assembled in this form, even though the involved arguments are well known (see for example \cite[Remark~10.30]{villani2008optimal}).

Let $\pi\in\couple(\mu, \nu)$ denote an optimal transport plan under a continuous cost function $c$.
Recalling the properties of the $c$-transform, any Kantorovich potential $f\in\CCc(\mu, \nu)$ satisfies $f(x) + f^c(y) \le c(x, y)$ for all $(x,y)\in\XY$ with equality if $(x, y)\in\partial_cf$. Fixing $(x, y)\in\supp\,\pi\subset\partial_cf$, this implies
\begin{equation*}
  x' \mapsto f(x') - c(x', y)\qquad\text{is minimal at}~x' = x.
\end{equation*}
Therefore, if $\X$ is a smooth manifold (without boundary) and the functions $f$ as well as $x'\mapsto c(x', y)$ are both differentiable at $x$, it has to hold that
\begin{equation}\label{eq:equalgradients}
  \nabla f(x) = \nabla_x c(x, y),
\end{equation}
by which we mean equality of the respective gradients in charts of $\X$.
This relation determines the derivatives of Kantorovich potentials in the set $\pr_\X(\Gamma_{\pi, c})\subset\X$, where we define
\begin{equation}\label{eq:gamma}
  \Gamma_{\pi, c} = \big\{(x, y)\in\supp\,\pi\,\big|\,\nabla_x c(x, y)~\text{exists}\big\}.
\end{equation}
In order to conclude uniqueness of $f$ up to constants from characterization \eqref{eq:equalgradients}, we will make use of the following auxiliary result. 
A proof is provided in \cref{app:proofs}.
\begin{lemma}{}{lipschitzgradients}
  Let $M\subset \X$ be an open and connected subset of a smooth manifold $\X$ and let $f_1, f_2\colon M\to\RR$ be locally Lipschitz.
  If $\nabla f_1 = \nabla f_2$ on a set that has full Lebesgue measure in charts of $M$, then $f_1 - f_2$ is constant on $M$.
\end{lemma}
At this point, several considerations have to be taken into account.
\begin{enumerate}
  \item The region $M \subset \X$ chosen for \cref{lem:lipschitzgradients} should have full $\mu$-measure, or it should at least be possible to uniquely recover the function $f$ from $f|_M$ on a set with full $\mu$-measure.
  \item The Kantorovich potential $f$ must be locally Lipschitz on $M$ in order for \cref{lem:lipschitzgradients} to be applicable.
    This also implies the existence of $\nabla f$ in a set of full Lebesgue measure in each chart (via Rademacher's theorem).
  \item The cost function has to be sufficiently smooth in its first argument along the transport.
    More precisely, $\pr_\X(\Gamma_{\pi, c})$ must have full Lebesgue measure in charts of $M$.
    Otherwise, relation~\eqref{eq:equalgradients} would not determine the gradients of $f$ suitably for application of \cref{lem:lipschitzgradients}.
\end{enumerate}

One immediate conclusion of the first point is the necessity of the condition $\closure(M) = \supp\,\mu$, without which some mass would be out of reach from the region $M$ controlled by the gradients.
The second and third points stress that the cost function should be locally Lipschitz in its first argument, in a way that is inherited to $\CCc$.
In order to choose $M$ properly for a general uniqueness statement, we recall the set $\Sigma_\pi \subset \supp\,\mu$ of points where induced regularity fails (\cref{def:regularity}).
We know that this set is closed (see \cref{lem:notpartiallycompact} in \cref{app:proofs}), that $\supp\,\mu \setminus \Sigma_\pi \subset \pr_\X(\supp\,\pi)$, and that $\Sigma_\pi$ contains all points of discontinuity of $f|_{\supp\,\mu}$ for any $f\in \CCc(\mu, \nu)$ (see \cref{lem:regularitycompact}).

\begin{theorem}{Uniqueness, connected support}{uniqueconnected}
  Let $X$ be a smooth manifold, $Y$ be Polish, $\mu\in\PC(\X)$, $\nu\in\PC(\Y)$, and $c\colon\XY\to\RR_+$ continuous such that $c(\cdot, y)$ is locally Lipschitz locally uniformly in $y\in\Y$ with $\ot_c(\mu, \nu) < \infty$.
  Let $\Sigma_\pi\subset\supp\,\mu$ and $\Gamma_{\pi, c}\subset\supp\,\pi$ be as in equations \eqref{eq:sigma} and \eqref{eq:gamma}  for an optimal $\pi \in \couple(\mu, \nu)$.
  If
  \begin{enumerate}[topsep=1ex, parsep=0.5ex, font={\normalfont\color{orange!45!black}}, label=\theenumi)]
    \item $\mu(\supp\,\mu \setminus \Sigma_\pi) = 1$,
    \item $M = \interior(\supp\,\mu \setminus \Sigma_\pi)$ is connected with
      $\closure(M) = \supp\,\mu$, 
    \item $\pr_\X(\Gamma_{\pi, c})$ has full Lebesgue measure in charts of
      $M$,
  \end{enumerate}
  then the Kantorovich potentials $\CCc(\mu, \nu)$ are unique.
\end{theorem}

\vspace{-1ex}

\begin{remark}{Uniqueness of optimal transport maps}{transportmap}
  Solving equation~\eqref{eq:equalgradients} for $y\in\Y$ is a standard method to construct and study optimal transport maps $\otmap\colon\X\to\Y$, see \cite[Chapter~10]{villani2008optimal}.
  In this context, a natural requirement is the injectivity of $\nabla_x c(x, \cdot)$, denoted as the \emph{twist condition}, which also implies that an optimal map is uniquely defined wherever \eqref{eq:equalgradients} holds.
  To make sure that $\otmap$ is determined by \eqref{eq:equalgradients} $\mu$-almost surely, one usually imposes some form of regularity on $c$ (such as local semiconcavity) and requires the probability measure $\mu$ to assign no mass to sets on which Kantorovich potentials may be non-differentiable (e.g., by assuming a Lebesgue density).
  \Cref{thm:uniqueconnected} helps clarify to which extent similar assumptions on $c$ and $\mu$ are necessary if we are only interested in uniqueness of the Kantorovich potentials, and not in the uniqueness of optimal maps.
\end{remark}

\vspace{-1ex}

\begin{proof}
  According to \cref{lem:regularitycompact}, each point of $M$ admits an open neighborhood $U$ and a compactum $K\subset \Y$ such that $f(x) = \inf_{y\in K} c(x, y) - f^c(y)$ for all $x\in U$ and $f\in\CCc(\mu, \nu)$.
  Since $c(\cdot, y)$ is locally Lipschitz uniformly in $y\in K$ by assumption, we conclude that $f$ is locally Lipschitz on $M$.
  Now, let $f_1, f_2\in\CCc(\mu, \nu)$ and let $A$ be the subset of $M$ where both functions are differentiable.
  Set $B = A \cap \pr_\X(\Gamma_{\pi, c})$, which is the set where the gradients of $f_1$ and $f_2$ have to coincide via~\eqref{eq:equalgradients}.
  Due to Rademacher's Theorem (see, e.g., \cite[Theorem~3.1.6]{federer2014geometric}) and assumption~3, the set $B$ has full Lebesgue measure in each chart of $M$.
  We can thus apply \cref{lem:lipschitzgradients} under assumption~2 and conclude that $f_1 = f_2$ (up to a constant) on $M$.
  Since any $f\in\CCc(\mu, \nu)$ is continuous at each point in $\supp\,\mu\setminus\Sigma_\pi \subset \closure(M)$ via \cref{lem:regularitycompact}, the function $f$ is uniquely determined on $\supp\,\mu \setminus\Sigma_\pi$ by its values on $M$.
  This shows that $f_1 = f_2$ (up to a constant) on $\supp\,\mu \setminus \Sigma_\pi$, which is a set of full $\mu$-measure by condition~1.
\end{proof}

Without further context, the assumptions in this theorem may seem to be fairly opaque, as they rely on specific details of an optimal transport plan $\pi$, like the $\mu$-measure of $\Sigma_\pi$ and topological properties of $\supp\,\mu \setminus \Sigma_\pi$.
In more specialized settings, however, the requirements of \cref{thm:uniqueconnected} can often be checked easily.
As a first example, we assume $Y$ to be compact.
If $c$ is differentiable in its first component, then all assumptions of \cref{thm:uniqueconnected} that depend on $\pi$ are automatically satisfied and only conditions on the topology of $\supp\,\mu$ remain.

\begin{corollary}{}{uniqueconnectedcompact}
  Let $X$ be a smooth manifold, $Y$ be compact Polish, $\mu\in\PC(\X)$, $\nu\in\PC(\Y)$, and $c\colon\XY\to\RR_+$ continuous such that $c(\cdot, y)$ is differentiable and locally Lipschitz uniformly in $y\in\Y$ with $\ot_c(\mu, \nu) < \infty$.
  If $M = \interior(\supp\,\mu)$ is connected and $\closure(M) = \supp\,\mu$, then the Kantorovich potentials $\CCc(\mu, \nu)$ are almost surely unique.
\end{corollary}

\begin{proof}
  For compact $Y$, it holds by definition that $\Sigma_\pi = \emptyset$.
  Thus, conditions~1 and 2 of \cref{thm:uniqueconnected} are satisfied.
  Condition~3 is ensured since $c$ is differentiable in the first component and we thus find $\Gamma_{\pi, c} = \pr_\X(\supp\,\pi) = \supp\,\mu$ (since the projection $\pr_\X$ is a closed map for compact $\Y$).
\end{proof}

If $Y$ is not compact, uniqueness statements based on \cref{thm:uniqueconnected} hinge on the behavior of the cost function outside of $Y$-compacta.
The simplest setting of this kind is the one where $c(\cdot, y)$ is (locally) Lipschitz uniformly in $y\in Y$.
Then, a statement analogous to \cref{cor:uniqueconnectedcompact} is possible, where the only obstacle is to assert condition~3 of \cref{thm:uniqueconnected}.
To provide a convenient sufficient criterion to this end, we work with the assumption that
\begin{equation}\label{eq:absolutecontinuity}
  \lambda ~\text{is absolutely continuous w.r.t.}~\varphi_\#\mu~
  \text{on}~\range\,\varphi
\end{equation}
for any chart $\varphi$ of $M = \interior(\supp\,\mu)$, where  $\varphi_\#\mu\coloneqq \mu\circ \varphi^{-1}$ corresponds to the push-forward  measure of $\mu$ under $\varphi$ and $\lambda$ denotes the Lebesgue measure.
Loosely speaking, this property states that the mass of $\mu$ can be placed quite arbitrarily on $\supp\,\mu$, as long as it contains a continuous component everywhere on its support.
As an example where condition~\eqref{eq:absolutecontinuity} fails for $\X = \RR$, consider a measure $\mu$ that is concentrated on the rational numbers $\mathbb{Q}$ but where $\supp\,\mu$ contains an open set.

\begin{corollary}{}{uniqueconnectedlipschitz}
  Let $X$ be a smooth manifold, $Y$ Polish, $\mu\in\PC(\X)$, $\nu\in\PC(\Y)$, and $c\colon\XY\to\RR_+$ continuous such that $c(\cdot, y)$ is differentiable and locally Lipschitz uniformly in $y\in\Y$ with $\ot_c(\mu, \nu) < \infty$.
  If $M = \interior(\supp\,\mu)$ is connected such that $\closure(M) = \supp\,\mu$ and condition \eqref{eq:absolutecontinuity} holds, then the Kantorovich potentials $\CCc(\mu, \nu)$ are unique.
\end{corollary}
\begin{proof}
  Since $c(\cdot, y)$ is assumed to be locally Lipschitz uniformly in $y\in\Y$, every Kantorovich potential $f\in\CCc(\mu, \nu)$ is locally Lipschitz on $\supp\,\mu$.
  Looking at the proof of \cref{thm:uniqueconnected}, we furthermore note that the set $\Sigma_\pi$ can actually be replaced by any other subset of $\X$ that contains all points at which some Kantorovich potential fails to be locally Lipschitz.
  Therefore, we may functionally assume $\Sigma_\pi = \emptyset$ in conditions~1 and~2 of \cref{thm:uniqueconnected}, and only have to show that condition~3 holds with $M = \interior(\supp\,\mu)$.
  Since $\Gamma_{\pi, c} = \supp\,\pi$ due to differentiability of $c$, we find a Borel set $A \subset \pr_\X(\Gamma_{\pi, c}) \subset \X$ that satisfies $\mu(A) = 1$ (\cref{lem:measurability}).
  Hence, $\range\,\varphi \setminus \varphi(A \cap \domain\,\varphi)$ is a $\varphi_\#\mu$-null set for any chart $\varphi$ of $M$.
  Due to condition~\eqref{eq:absolutecontinuity}, it is also a Lebesgue-null set and we conclude that $\pr_\X(\Gamma_{\pi, c})$ has full Lebesgue measure in charts of $M$.
\end{proof}

\begin{remark}{Local semiconcavity}{semiconcave}
  In \cref{ex:continuityconcave}, we pointed out that Kantorovich potentials can inherit semiconcavity from the cost function.
  In fact, it is possible to formulate \cref{cor:uniqueconnectedlipschitz} for costs where $c(\cdot, y)$ is locally semiconcave (instead of locally Lipschitz) uniformly in $y\in\Y$, in the sense of equation~\eqref{eq:semiconcave}.
  This alternative formulation, whose proof is documented in \cref{app:proofs}, can be put to use in settings where \cref{cor:uniqueconnectedlipschitz} might not apply directly.
  For example, let $X$ and $Y$ be two (possibly distinct) affine subspaces of $\RR^d$ equipped with an atlas of linear charts.
  Then the squared Euclidean cost function $c(x, y) = \|x - y\|^2$ is differentiable and locally semiconcave (but \emph{not} necessarily locally Lipschitz) in $x$ uniformly in $y$.
  Consequently, the Kantorovich potentials in this setting are unique under the mild conditions on $\mu$ imposed by \cref{cor:uniqueconnectedlipschitz}.
  If $\supp\,\mu$ and $\supp\,\nu$ are separated sets, the same holds for Euclidean cost functions $c(x,y)=\|x-y\|^p$ with $0 < p \le 2$.
\end{remark}

One question largely unaddressed by the previous results is how \cref{thm:uniqueconnected} fares in the general setting that $c(\cdot, y)$ is actually no more than locally Lipschitz \emph{locally} uniformly in $y\in Y$, which is typically the case for rapidly growing cost functions.
In the next section, we show that such cost functions often confine the set $\Sigma_\pi$ to the boundary of $\supp\,\mu$, which makes the application of \cref{thm:uniqueconnected} particularly simple: condition~1 collapses into $\mu(\partial\,\supp\,\mu) = 0$ and condition~2 only relies on the topology of $\supp\,\mu$.
Furthermore, condition~3 is always satisfied when $c$ is differentiable in the first component, since then $\Gamma_{\pi, c} = \supp\,\pi$ and $M \subset \supp\,\mu \setminus \Sigma_\pi \subset \pr_\X(\Gamma_{\pi, c})$ via \cref{lem:regularitycompact}.
This leads to the following statement.

\begin{corollary}{}{uniqueconnectedboundary}
  Let $X$ be a smooth manifold, $Y$ Polish, $\mu\in\PC(\X)$, $\nu\in\PC(\Y)$, and $c\colon\XY\to\RR_+$ continuous such that $c(\cdot, y)$ is differentiable and locally Lipschitz locally uniformly in $y\in\Y$ with $\ot_c(\mu, \nu) < \infty$.
  If $\Sigma_\pi \subset \partial\,\supp\,\mu$ for an optimal plan $\pi\in\couple(\mu, \nu)$ and if  $M = \interior(\supp\,\mu)$ is connected such that $\closure(M) = \supp\,\mu$ and $\mu(M) = 1$, then the Kantorovich potentials $\CCc(\mu, \nu)$ are unique.
\end{corollary}

\section{Interior regularity}
\label{sec:interior}

We now investigate conditions on the cost function $c$ under which each optimal transport plan $\pi$ induces regularity in the interior of the support of $\mu$.
In other words, we search for criteria that ensure
\begin{equation}\label{eq:sigmaboundary}
  \Sigma_\pi \subset \partial\,\supp\,\mu,
\end{equation}
where $\Sigma_\pi\subset\supp\,\mu$ was defined in \eqref{eq:sigma} and denotes the set of points where $\pi$ fails to induce regularity.
This property is of interest for both \cref{thm:uniquedisconnected} (applied to $\nu$ in place of $\mu$) and \cref{thm:uniqueconnected} (in the form of \cref{cor:uniqueconnectedboundary}), as it guarantees continuity of Kantorovich potentials in the interior of the support.
One possibility to show \eqref{eq:sigmaboundary} is to rule out that points from the interior of $\supp\,\mu$ can be \enquote{sent towards infinity} via any optimal plan.
In \cref{lem:interiorregularity} below, we demonstrate this to be true for a broad class of (asymptotically) quickly growing cost functions.
In fact, \eqref{eq:sigmaboundary} will often fail if the costs are not sufficiently quickly growing (in this case, however, \cref{cor:uniqueconnectedlipschitz} will often be applicable to conclude uniqueness of the potentials).
As preparation, we introduce some notation.

\begin{definition}{}{}
  Let $(y_n)_{n\in\NN}\subset Y$ be a sequence.
  We write $y_n \to \infty$ (as $n\to\infty$) if each compact set in $\Y$ contains only a finite number of elements.
\end{definition}

Despite the suggestive notation, note that $y\to\infty$ does not necessarily mean that $y$ leaves all bounded sets if $\Y$ is not a proper metric space, i.e., if not all bounded closed sets in $\Y$ are compact.

\begin{definition}{}{}
  For any $(x, y)\in\XY$, define the \emph{region of dominated cost}
  \begin{equation}\label{eq:dominatedcost}
    C(x, y) = \big\{x'\in\X\,\big|\,c(x', y) \le c(x, y)\big\}.
  \end{equation}
  For a given sequence $(x_n, y_n)_{n\in\NN}\subset\XY$, also define $C_\infty = \limsup_{n\to\infty} C(x_n, y_n)$.
\end{definition}

In Euclidean settings, for example, under costs $c(x,y) = \|x - y\|^p$ for $x,y\in\RR^d$ and $p \geq 1$, the set $C(x, y)$ is the closed ball centered at $y$ with radius $\|x - y\|$.
As we will see next, the geometry of the asymptotic set $C(x, y)$ as $y\to\infty$ shapes the region where Kantorovich potentials do not attain finite values for quickly growing costs.

\begin{lemma}{Interior regularity}{interiorregularity}
  Let $\X$ and $\Y$ be Polish, $c\colon\XY\to\RR_+$ be continuous, and $\mu\in\PC(\X)$, $\nu\in\PC(\Y)$ with $\ot_c(\mu, \nu) < \infty$.
  Let $\pi\in\couple(\mu, \nu)$ be optimal and $x\in\supp\,\mu$ such that there exists $(x_n, y_n)_{n\in\NN}\subset\supp\,\pi$ with $x_n\to x$ and $y_n\to\infty$ as $n\to\infty$.
  If
  \begin{enumerate}[topsep=1ex, parsep=0.5ex, font={\normalfont\color{green!30!black}}, label=\theenumi)]
    \item there is $(\tilde{x}_n)_{n\in\NN}\subset \X$ converging to $x$ with $\lim_{n\to\infty} c(x_n, y_n) - c(\tilde{x}_n, y_n) = \infty$,
  \end{enumerate}
  then all $f\in\CCc(\mu, \nu)$ assume the value $-\infty$ on $C_\infty = \limsup_{n\to\infty} C(\tilde{x}_n, y_n)$.
  If additionally
  \begin{enumerate}[resume, topsep=1ex, parsep=0.5ex, font={\normalfont\color{green!30!black}}, label=\theenumi)]
    \item $C_\infty$ contains an open subset $U$ that touches $x$, meaning $x\in \closure(U)$,
  \end{enumerate}
  then $x\in\partial\,\supp\,\mu$.
\end{lemma}

\begin{proof}
  Let $x'\in C_\infty$. After a suitable subsequence has been taken, we may assume that $x'\in C(\tilde{x}_n, y_n)$ for all $n\in\NN$.
  For $f\in\CCc(\mu, \nu)$, note that $f(x_n) = c(x_n, y_n) - f^c(y_n)$ and observe
  \begin{align*}
    f(x')
    &\le
    c(x', y_n) - f^c(y_n) \\
    &\le
    c(\tilde{x}_n, y_n) - f^c(y_n) \\
    &=
    c(\tilde{x}_n, y_n) - c(x_n, y_n) + f(x_n)
    \to
    -\infty
  \end{align*}
  due to condition~1 and the upper-semicontinuity of $f$, which implies $\sup_{n\in\NN} f(x_n) < \infty$.
  This shows the first claim.
  To show the second claim, we lead $x\in\interior(\supp\,\mu) \cap \closure(U)$ to a contradiction: by density of $\pr_\X(\supp\,\pi)$ in $\supp\,\mu$, such an $x$ would imply $\pr_\X(\supp\,\pi) \cap U \neq \emptyset$.
  Thus, there would exist $(x',y')\in\supp\,\pi \subset \partial_c f$ with $-\infty = f(x') = c(x', y') - f^c(y') > -\infty$.
\end{proof}

The assumptions and interpretation of \cref{lem:interiorregularity} deserve some more commentary.
\begin{enumerate}
\item
  By definition, any $x\in\Sigma_\pi$ admits a suitable sequence $(x_n, y_n)_{n\in\NN}\subset\supp\,\pi$ with $x_n\to x$ and $y_n\to\infty$, meaning that \cref{lem:interiorregularity} can be applied to all points at which $\pi$ fails to induce regularity.
  If conditions 1 and 2 are shown to hold for such sequences, the desired result, $\Sigma_\pi \subset \partial\,\supp\,\mu$, immediately follows.
\item
  \cref{lem:interiorregularity} is formulated with very specific sequences $(x_n, y_n)_{n\in\NN}$ in mind, namely those that live on the support of an optimal plan $\pi$.
  However, for many cost functions $c$ of interest, conditions~1 and 2 can be shown to hold for \emph{all} sequences $(x_n, y_n)_{n\in\NN} \subset \XY$ where $x_n$ converges and $y_n\to\infty$.
  For example, condition~1 is generically implied for every $c$ that satisfies the abstract growth property~$(\text{H}_\infty)_2$ in \textcite[Chapter~10]{villani2008optimal}.
  Thus, explicit knowledge about $\mu$, $\nu$, or $\pi$ is often not necessary to conclude $\Sigma_\pi \subset \partial\,\supp\,\mu$.
  \item
  The claim of \cref{lem:interiorregularity}, which relates transport towards infinity with the divergence of Kantorovich potentials, can be reversed in a certain sense.
  Fix any $x\in\X$, not necessarily in $\supp\,\mu$.
  If $f(x) = -\infty$, this implies $\inf_{y \in \Y} c(x,y)-f^c(y)= -\infty$.
  It follows that there exists $(y_n)_{n\in\NN}$ with $c(x, y_n) - f^c(y_n) \to -\infty$ as $n\to\infty$.
  Due to the upper-semicontinuity of $f^c$, we conclude $y_n\to\infty$.
  It is clear that $f$ then attains the value $-\infty$ on the full set $C_\infty = \limsup_{n\to\infty} C(x, y_n)$.
  If condition~2 of \cref{lem:interiorregularity} holds, this set has to be disjoint from the interior of $\supp\,\mu$.
\end{enumerate}

\begin{example}{}{innereuclidean}
  Let $\X = \Y = \RR^d$ with squared Euclidean costs $c(x, y) = \|x - y\|^2$.
  Given any sequence $x_n\to x\in\RR^d$ and $y_n\to\infty$, the choice $\tilde{x}_n = x_n + (y_n - x_n)/\|y_n - x_n\|^{3/2}$ provides a perturbation of $x_n$ that satisfies condition~1 in \cref{lem:interiorregularity}.
  To see that condition~2 is also satisfied, note that the asymptotic set $C_\infty$ will contain an open half-space anchored at the point $x$  (see \cref{fig:interiorregularity}a for an illustration).
  The (inwards pointing) normal direction is given by an (arbitrary) limit point $u$ of the directions $u_n = (y_n - \tilde{x}_n)/\|y_n - \tilde{x}_n\|$.
  Such a limit exists, since the unit sphere in $\RR^d$ is compact.
  Therefore, \cref{lem:interiorregularity} lets us conclude property \eqref{eq:sigmaboundary} and provides additional insights about the set $\{f = -\infty\}$ and its relation to the direction $u$ of transport towards infinity.
  Indeed, the fact that $C_\infty$ always contains half-spaces (which is also true for all other $l_p$ costs for $p > 1$, but not necessarily for $p = 1$, see \cref{fig:interiorregularity}b) implies that the interior of the convex hull of $\supp\,\mu$ is contained in the set $\{f > - \infty\}$.
\end{example}

\begin{figure}
  \centering\footnotesize
  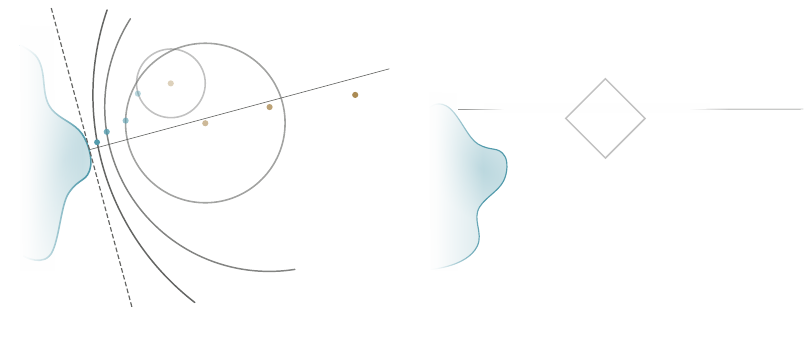
  \caption{Asymptotic regions of dominated cost.
    Both figures depict the sets $C(\tilde{x_n}, y_n)$ as defined in \eqref{eq:dominatedcost} for a sequence $(\tilde{x}_n, y_n)_{n\in\NN}$ with $\tilde{x}_n\to x$ and $y_n\to\infty$ as $n\to\infty$, where $u = \lim_{n\to\infty} (y_n - \tilde{x}_n)/\|y_n - \tilde{x}_n\|$.
    In sketch \textbf{(a)}, the cost function $c(x, y) = h\big(\|x - y\|\big)$ for the Euclidean norm and some strictly increasing $h$ is chosen, while sketch \textbf{(b)} shows a similar setting for costs based on the $l_1$ norm.
    In both examples, condition~1 of \cref{lem:interiorregularity} is always satisfied if $h$ is differentiable and $h'(a) \to\infty$ as $a\to\infty$ (see \cref{thm:interiorregularity}).
    Therefore, all Kantorovich potentials assume the value $-\infty$ on the region $C_\infty$ if mass is transported from $x$ towards infinity in direction $u$.
    For illustration, \textbf{(b)} shows the rather exceptional case where $C_\infty$ is \emph{not} a half space, which (in this example) can only happen if $u$ is aligned with one of the coordinate axes.
    More precisely, the depicted shape is only possible if the vertical coordinates of $y_n$ converge to the vertical coordinate of the apex of $C_\infty$.
  }
  \label{fig:interiorregularity}
\end{figure}

The reasoning displayed in \cref{ex:innereuclidean} can be extended to other costs on $\RR^d$.
In particular, it is possible to treat certain families of translation invariant costs $c(x, y) = h(x - y)$ that were investigated in the seminal analysis of the regularity of $c$-concave functions by \textcite{gangbo1996geometry}.
Their assumption (H2), the \emph{cone condition}, is frequently employed and widely re-cited.
We work with a slightly relaxed version of this condition.
\begin{itemize}
  \item[(C)]
  Let $\sphericalangle\colon\RR^d\times\RR^d \to [0, \pi]$ denote the angle between two vectors.
  Assume there exists a height $s > 0$ and an angle $\vartheta\in(0, \pi)$ such that each $p\in\RR^d$ far enough from the origin admits a direction $u\in\RR^d$, $\|u\| = 1$, with associated cone
  \begin{equation*}
    K_u
    =
    \big\{q\in\RR^d\,\big|\, \sphericalangle(u, q) \le \vartheta/2, \langle u, q\rangle \le s \big\},
  \end{equation*}
  for which $h(p) \ge h(p + q)$ for all $q\in K_u$.
\end{itemize}

The difference between (C) and (H2) is that the former only requires the existence of suitable cones for a \emph{fixed} height $s$ and angle $\vartheta$, and not for increasingly large values of $s$ and $\vartheta$.
This relaxation is sufficient for our purposes, since it already implies condition 2 in \cref{lem:interiorregularity} for arbitrary sequences.
Examples for cost functions that satisfy (C) but not (H2) are $h(x) = \|x\|_1^p$ or $h(x) = \|x\|_\infty^p$ with $p > 0$.

Besides the cone condition (H2), \textcite{gangbo1996geometry} place further restrictions on $h$ for a number of important conclusions.
For instance, their Theorem~3.3, which characterizes properties of $c$-transformed functions for quickly growing costs, requires superlinear growth $\lim_{x\to\infty} h(x) / \|x\| = \infty$ (H3) and convexity of $h$ (H4).
Adapting the requirements of \cref{lem:interiorregularity} to translation invariant costs in Euclidean spaces, we can propose a weaker assumption.
\begin{itemize}
  \item[(G)]
  Assume each $x\in\RR^d$ admits $\delta_x\in\RR^d$ with $h(x) - h(x + \delta_x) \to \infty$ and $\delta_x \to 0$ as $x\to\infty$.
\end{itemize}

\begin{lemma}{}{assumptionG}
  Let $h\colon\RR^d\to\RR_+$ be locally Lipschitz.
  Assumption (G) is implied by either of the following conditions.
  \begin{enumerate}[topsep=1ex, parsep=0.5ex, font={\normalfont\color{green!30!black}},label=\theenumi)]
    \item $h$ is convex and satisfies $\lim_{x\to\infty} h(x) / \|x\| = \infty$.
    \item $h(x) = g(\|x\|)$ for $g\colon\RR_+ \to \RR_+$ locally Lipschitz and $\lim_{a\to\infty} g'(a) = \infty$ whenever $g'$ exists.
    \item For $h_u = a \mapsto h(au)$ with $u\in\RR^d$ it holds that $\lim_{a\to\infty}\inf_{\|u\| = 1, b \ge a} h_u'(b) = \infty$ whenever $h_u'$ exists.
  \end{enumerate}
\end{lemma}
The third of these criteria states that the radial derivatives of $h$ uniformly diverge to infinity when the radius is increased.
In fact, condition 1 and condition 2 both imply condition 3, which in turn implies (G).
The proof is deferred to the appendix.
With this sufficient growth condition for the costs in place, we are now prepared to show the following regularity result.

\begin{theorem}{Interior regularity, Euclidean spaces}{euclideaninteriorregularity}
  Let $\X = \Y = \RR^d$ for $d\in\NN$ and consider costs $c\colon\RR^d\times\RR^d\to\RR_+$ of the form $c(x, y) = h(x - y)$ for $h\colon\RR^d\to\RR_+$.
  If assumptions (C) and (G) hold, then 
  \begin{equation*}
    \Sigma_\pi\subset\partial\,\supp\,\mu
  \end{equation*}
  for any $\mu,\nu\in\PC(\X)$ with $\ot_c(\mu, \nu) < \infty$, where $\Sigma_\pi$ is defined in \eqref{eq:sigma} for $\pi\in \couple(\mu, \nu)$ optimal.
\end{theorem}

\begin{proof}
  Assume that $x_n \to x\in\RR^d$ and $y_n \to \infty$.
  Set $z_n = x_n - y_n$.
  To show condition~1 of \cref{lem:interiorregularity}, define $\tilde{x}_n = x_n + \delta_{z_n}$ with $\delta_{z_n}$ provided by assumption~(G).
  Since $z_n\to\infty$, we find $\delta_{z_n} \to 0$ and $\tilde{x}_n \to x$.
  Furthermore, we observe
  \begin{equation*}
    c(x_n, y_n) - c(\tilde{x}_n, y_n)
    =
    h(z_n) - h(z_n + \delta_{z_n})
    \to \infty
  \end{equation*}
  as $n\to\infty$.
  To verify condition 2, let $p_n = \tilde{x}_n - y_n$.
  For $n$ sufficiently large, assumption~(C) guarantees the existence of directions $u_n$ such that $h(p_n) \ge h(p_n + q)$ for each $q\in K_{u_n}$.
  By picking a limit direction $u$ and selecting a suitable subsequence, we may assume $u_n \to u$.
  Let $U = \interior(x + K_u)$.
  We find $x\in\closure(U)$.
  Furthermore, it holds that $U \subset C_\infty = \limsup_{n\to\infty} C(\tilde{x}_n, y_n)$, which concludes the proof.
  To see this, one only needs to note that $\tilde{x}_n + K_{u_n} \subset C(\tilde{x}_n, y_n)$, since
  \begin{equation*}
    c(\tilde{x}_n, y_n)
    =
    h(\tilde{x}_n - y_n)
    \ge
    h(\tilde{x}_n - y_n + q)
    =
    c(\tilde{x}_n + q, y_n)
  \end{equation*}
  for all $q\in K_{u_n}$, and that $U \subset \limsup_{n\to\infty} \tilde{x}_n + K_{u_n}$ due to $\tilde{x}_n \to x$ and $u_n\to u$.
\end{proof}

Combining the previous result with \cref{cor:uniqueconnectedboundary} yields novel uniqueness statements for costs $c(x, y) = h(\|x-y\|)$ in Euclidean spaces.
In particular, it establishes that convexity of the function $h$ is not necessary for the uniqueness of Kantorovich potentials to be guaranteed.
This extends upon Proposition~B.2(b) of \textcite{bernton2021entropic} and Corollary~2.7 of \textcite{del2021central}, where convexity of $h$ is explicitly imposed.

\begin{remark}{\cite{gangbo1996geometry}, Theorem~3.3}{}
  In the proof of \cref{thm:euclideaninteriorregularity}, Assumption~(C) asserts that the set $C_\infty$ at $x\in\supp\,\mu$ includes the interior of a suitably shifted cone.
  If the stronger condition (H2) from \textcite{gangbo1996geometry} is assumed, the same argumentation can be repeated for progressively larger cone heights $s > 0$ and angles $\vartheta\in(0, \pi)$.
  This reveals that $C_\infty$ actually contains a half-space attached to the anchor point $x$, similar to \cref{ex:innereuclidean}.
  Consequently, the stronger cone condition (H2) implies that the set $\{f > -\infty\}$ for any Kantorovich potential $f\in\CC_c(\mu, \nu)$ contains the interior of the convex hull of the support of $\mu$.
  We furthermore conclude by \cref{lem:regularity} that $f$ is locally Lipschitz on the interior of the support, and that it is even locally semiconcave if $h$ is locally semiconcave.
  In this sense, \cref{thm:euclideaninteriorregularity} with (H2) generalizes Theorem~3.3 by \textcite{gangbo1996geometry} via weakening conditions (H3) and (H4) to (G).
\end{remark}

\cref{lem:interiorregularity} can also be applied in more general, non-Euclidean scenarios.
As an example, the following result establishes sufficient conditions for the validity of \eqref{eq:sigmaboundary} in geodesic spaces.
The proof strategy closely resembles the approach in \cref{thm:euclideaninteriorregularity}.
A central argument is the compactness of closed balls, which guarantees the existence of asymptotic directions of transport towards infinity.

\begin{theorem}{Interior regularity, geodesic spaces}{interiorregularity}
  Let $\X = \Y$ be a locally compact complete geodesic space with metric $d$ and let $c\colon\X^2\to\RR_+$ be of the form $c(x, y) = h\big(d(x, y)\big)$, where $h\colon\RR_+ \to \RR_+$ is locally Lipschitz with $\lim_{a\to\infty} h'(a) = \infty$ whenever $h'$ exists.
  Then
  \begin{equation*}
    \Sigma_\pi\subset\partial\,\supp\,\mu
  \end{equation*}
  for any $\mu,\nu\in\PC(\X)$ with $\ot_c(\mu, \nu) < \infty$, where $\Sigma_\pi$ is defined in \eqref{eq:sigma} for $\pi\in \couple(\mu, \nu)$ optimal.
\end{theorem}
\begin{proof}
  Let $x_n\to x\in\X$ and $y_n\to\infty$.
  Since $\X$ is a proper metric space (e.g., by the Hopf-Rinow theorem as stated in \cite[Proposition~3.7]{bridson2013metric}), each closed ball in $\X$ is compact.
  This implies $r_n = d(x_n, y_n)\to\infty$.
  Next, let $g\colon\RR_+\to\RR$ be defined by $g(a) = \inf_{b\ge a} h'(b)$, where the infimum is taken over all $b$ for which $h'$ exists.
  This function is non-decreasing and satisfies $g \le h'$ and $g(a)\to\infty$ as $a\to\infty$.
  We also let $\gamma_{x' x''}\colon[0, d(x', x'')]\to \X$ denote a geodesic connecting $x'$ to $x''$ in $\X$.

  To show condition~1 of \cref{lem:interiorregularity}, let $\tilde{x}_n = \gamma_{x_n y_n}(t_n)$ for $0 < t_n < 1$, meaning that $x_n$ is pushed towards $y_n$ along a geodesic to generate perturbations $\tilde{x}_n$.
  The amount $t_n$ by which $x_n$ is pushed is chosen to satisfy $t_n\to 0$ and $t_n\,g(r_n - 1) \to \infty$ as $n\to\infty$.
  Since $d(x, \tilde{x}_n) \le d(x, x_n) + t_n$, the points $\tilde{x}_n$ indeed converge to $x$.
  We also note that  $d(\tilde{x}_n, y_n) = r_n - t_n$ and observe, as $n\to\infty$,
  \begin{equation*}
    c(x_n, y_n) - c(\tilde{x}_n, y_n)
    =
    h(r_n) -  h(r_n - t_n)
    =
    \int_{-t_n}^{0} h'(r_n + t) \dif t
    \ge
    t_n\,g(r_n - 1)
    \to
    \infty.
  \end{equation*}
  To verify condition~2, let $u_n = \gamma_{\tilde{x}_ny_n}(1)$ and pick a limit point $u\in\X$ of this sequence.
  Such a point exists, since the points $u_n$ are bounded and $\X$ is proper.
  By selecting a suitable subsequence, we may assume $u_n \to u$.
  Let $U$ be the open unit ball at $u$.
  We will show $U\subset \limsup C(\tilde{x}_n, y_n)$ and $x\in\closure(U)$.
  The latter is evident, since $d(x, u) \le d(x, \tilde{x}_n) + d(\tilde{x}_n, u_n) + d(u_n, u) \to 1$ as $n\to\infty$.
  To see the former, let $x'\in U$ and note that $d(x', u_n) < 1 = d(\tilde{x}_n, u_n)$ for large $n$.
  Then
  \begin{equation*}
    d(x', y_n)
    \le
    d(x', u_n) + d(u_n, y_n)
    <
    d(\tilde{x}_n, u_n) + d(u_n, y_n)
    =
    d(\tilde{x}_n, y_n).
  \end{equation*}
  For $n$ large enough such that $h$ can be assumed to be  non-decreasing, this implies $h\big(d(x', y_n)\big) \le h\big(d(\tilde{x}_n, y_n)\big)$ and thus $x'\in C(\tilde{x}_n, y_n)$.
\end{proof}

The proof above shows that the asymptotic region $C_\infty$ in \cref{lem:interiorregularity} at the very least contains the unit ball touching $x$ centered at a suitable $u\in\X$.
This argument can trivially be extended to balls of arbitrary radius $r > 1$, which provides additional insight about the shape of $C_\infty$.


\section{Proofs of the main results}
\label{sec:uniquenessproof}

In the following, we formulate the proofs of the uniqueness statements for Kantorovich potentials under disconnected support, \cref{thm:uniquedisconnected} and \cref{cor:uniquedisconnected}.
For reasons of exposition, we start with \cref{thm:uniquedisconnected} under continuity assumption~2, before we document the adjustments necessary to prove the theorem under the alternative assumption~1, which requires a slightly different strategy.
Afterwards, we provide the arguments to extend \cref{thm:uniquedisconnected} to countable $I$ as claimed in \cref{cor:uniquedisconnected}.
For notational convenience, we denote the topological closure of a set $A$ by $\overline{A}$ instead of $\closure(A)$ in this section.

\begin{proof}[Proof of \cref{thm:uniquedisconnected} under assumption 2]
Recall decomposition~\eqref{eq:components} of the support of $\mu$ and $\nu$ into connected components $(\X_i)_{i\in I}$ and $(\Y_j)_{j\in J}$ for finite $I$ and countable $J$. 
Since we consider the second condition of \cref{thm:uniquedisconnected} first, we can assume for each $f^c\in\CCcc(\mu, \nu)$ and $j\in J$ with $\nu(\Y_j) > 0$  that $\smash{f^c|_{\Y_j}}$ is continuous and that the set $\tilde\Y_j = \supp\,\nu|_{\Y_j}$ is connected.
As $I$ is finite, each $\X_i$ is open in $\supp\,\mu$ and consequently satisfies $\mu(\X_i) > 0$.
Therefore, the $\X_i$-restricted optimal transport problem is well defined and we can fix a representative $f_i\in\CC_{c_{\X_i}}(\mu_{\X_i}, \nu_{\X_i})$ for each $i\in I$.
The uniqueness assumption in \cref{thm:uniquedisconnected} together with \cref{lem:uniqueness} implies that $f_i$ is uniquely determined on $\pr_\X(\supp\,\pi)\cap \X_i$ (up to an additive offset), so its actual choice does not matter to us.
Applying \cref{lem:restriction}, we conclude that each $f\in\CCc(\mu, \nu)$ can be assigned a unique \emph{offset vector} $a = (a_i)_{i\in I} \in \RR^{|I|}$ with components $a_i = f(x_i) - f_i(x_i)$, where the point $x_i \in \pr_\X(\supp\,\pi)\cap \X_i$ can be chosen arbitrarily.
We suggestively write $f = f_a$ if $f\in\CCcc(\mu, \nu)$ has offset vector $a$ and emphasize that the equality
\begin{equation}\label{eq:uniquedecomposition}
  f_a = \sum_{i\in I} \ind_{\X_i} \cdot (f_i + a_i)
\end{equation}
holds on $\pr_\X(\supp\,\pi)$.
Clearly, two Kantorovich potentials in $\CCc(\mu, \nu)$ have identical offset vectors if and only if they coincide on $\pr_\X(\supp\,\pi)$.
Therefore, almost sure uniqueness of $\CCc(\mu, \nu)$ follows if we can show that there is only a single feasible offset vector $a$ (up to an additive constant that is the same in each component).
To formalize this idea, we divide the support of $\pi$ into closed disjoint pieces $\Gamma_{ij} = \supp\,\pi \cap (\X_i\times\Y_j)$, and we say that two indices $i_1$ and $i_2$ in $I$ are \emph{linked} if there exists a \emph{contact index} $j\in J$ with $\nu(\Y_j) > 0$ and a \emph{contact point} $y\in\Y_j$ so that
\begin{equation}\label{eq:linkedindices2}
  y
  \in
  \overline{\pr_\Y(\Gamma_{i_1j})} \cap \overline{\pr_\Y(\Gamma_{i_2j})}.
\end{equation}
Intuitively, two indices in $I$ are linked if the masses transported from $\X_{i_1}$ and $\X_{i_2}$ to $\Y_j$ touch one another at a common point $y\in\Y_j$. 
In a first step, we establish that
\begin{equation*}
  \text{$i_1$ and $i_2$ are linked}
  \qquad\text{implies}\qquad 
  \text{$a_{i_1} - a_{i_2}$ is fixed}
\end{equation*}
under the adopted continuity assumptions on $\CCcc(\mu, \nu)$, where the right hand side indicates that the difference $a_{i_1} - a_{i_2}$ has to be the same for all feasible offset vectors.
In a second step, we then show that non-degeneracy of the optimal plan $\pi$ guarantees that there are enough contact points to connect all indices in $I$, which will conclude the proof.

\paragraph{Step 1.}
Let $i_1$ and $i_2$ be indices in $I$ that are linked through a contact point $y\in Y_j$ for $j\in J$.
According to \eqref{eq:linkedindices2}, there are sequences $(x_n, y_n)_n \subset\Gamma_{i_1j}$ and $(x'_n, y'_n)_n\subset\Gamma_{i_2j}$ such that $y_n\to y$ and $y'_n\to y$ in $Y_j$ as $n\to\infty$.
Since $f^{}_a \oplus f_a^c = c$ on $\supp\,\pi$ as well as $f_a = f_{i_1} + a_{i_1}$ and $f_a = f_{i_2} + a_{i_2}$ on $\pr_\X(\Gamma_{i_1j})$ respectively $\pr_\X(\Gamma_{i_2j})$ due to relation~\eqref{eq:uniquedecomposition}, we find
\begin{align*}
  a_{i_1} - a_{i_2}
  &=
  \big(c(x_n, y_n) - f_{i_1}(x_n) - f_a^c(y_n)\big)
  - \big(c(x'_n, y'_n) - f_{i_2}(x'_n) - f_a^c(y'_n)\big)
\intertext{for all $n\in\NN$.
Exploiting the continuity of $f_a^c|_{Y_j}$ at the contact point $y\in Y_j$, we thus obtain}
  a_{i_1} - a_{i_2}
  &=
  \lim_{n\to\infty} \big(c(x_n, y_n) - f_{i_1}(x_n) - f_a^c(y_n)\big)
  - \big(c(x'_n, y'_n) - f_{i_2}(x'_n) - f_a^c(y'_n)\big) \\
  &=
  \lim_{n\to\infty} c(x_n, y_n) - c(x'_n, y'_n) - f_{i_1}(x_n) + f_{i_2}(x'_n).
\end{align*}
Crucially, the limit in the second line exists and does not depend on $a$ anymore.
It only depends on the cost function $c$, the restricted potentials $f_i$, and the sets $\Gamma_{ij}$ (determined by $\pi$), whose topologies decide the contact point $y$ and the involved sequences.
Hence, knowing the value of $a_{i_1}$ determines the one of $a_{i_2}$ and vice versa.

\paragraph{Step 2.}
It is left to show that all indices are linked, at least indirectly, such that the vector $a$ is in fact determined by fixing a single component.
To do so, we consider an arbitrary decomposition $I = I_1 \cup I_2$ of the index set $I$ into a disjoint union of non-empty subsets, and show that there always exist $i_1\in I_1$ and $i_2 \in I_2$ that are linked.
First, define
\begin{equation}\label{eq:DefinitionIndexsetJ1}
  J_1
  =
  \big\{j\in J \,\big|\, \pi(\Gamma_{ij}) > 0~\text{for some $i\in I_1$}\big\}
\end{equation}
and analogously $J_2$.
Intuitively, an index $j$ is in $J_1$ (or $J_2$) if $\pi$ transports mass between $Y_j$ and some component $X_i$ with $i\in I_1$ (or $i\in I_2$).
We note that the sets $J_1$ and $J_2$ cannot be disjoint: if they were, then $\pi$ would transport all mass in $\bigcup_{i\in I_1} \X_{i}$ to $\bigcup_{j\in J_1} \Y_{j}$ and vice versa, contradicting the condition of non-degeneracy.
Formally, this follows from $0 < \sum_{i\in I_1} \mu(\X_i) < 1$ and
\begin{align*}
  \sum_{i\in I_1} \mu(\X_{i})
  &=
  \sum_{i\in I_1}\pi\big(\supp\,\pi \cap (\X_i \times \Y)\big) \\
  \textcolor{black!75!white}{\small\text{(definition of $J_1$)}}\quad
  &=
  \sum_{i\in I_1}\sum_{j\in J_1} \pi\left(\Gamma_{ij}\right) \\
  \textcolor{black!75!white}{\small\text{($J_1$ and $J_2$ disjoint)}}\quad
  &=
  \sum_{j\in J_1}\pi\big(\supp\,\pi \cap (\X \cap \Y_j)\big)
  =
  \sum_{j\in J_1} \nu(\Y_j),
\end{align*}
the former of which holds since $I_1$ is nonempty and a proper subset of $I$.
Therefore, $J_1$ and $J_2$ are not disjoint and we find some $j\in J_1\cap J_2$.
By definition of $J_1$ and $J_2$, this implies that the two sets
\begin{equation}\label{eq:B1B2}
  B_1 = \bigcup_{i\in I_1} \overline{p_\Y(\Gamma_{ij})} \subset Y_j
  \qquad\text{and}\qquad
  B_2 = \bigcup_{i\in I_2} \overline{p_\Y(\Gamma_{ij})} \subset Y_j
\end{equation}
have positive $\nu$-mass and are thus non-empty.
Since $\pi\big(\bigcup_{i\in I}\Gamma_{ij}\big) = \pi(\X\times\Y_j) = \nu(\Y_j) > 0$, we can apply (a suitably restricted version of) \cref{lem:measurability} to conclude that $ p_\Y\big(\bigcup_{i\in I} \Gamma_{ij}\big) = \bigcup_{i\in I} p_\Y(\Gamma_{ij})$ contains a subset that is dense in $\tilde\Y_j = \supp\,\nu|_{\Y_j}$.
Thus, we observe
\begin{equation}\label{eq:YjB1B2}
  \tilde\Y_j
  \subset
  \overline{\bigcup_{i\in I} p_\Y(\Gamma_{ij})}
  =
  \overline{\bigcup_{i\in I_1} p_\Y(\Gamma_{ij})}
  \cup
  \overline{\bigcup_{i\in I_2} p_\Y(\Gamma_{ij})}
  =
  B_1 \cup B_2,
\end{equation}
where the last equality hinges on the fact that $B_1$ and $B_2$ are closed (this is where we need the assumption that $I$ is finite).
Since $\nu(B_k) = \nu|_{\Y_j}(B_k) > 0$, we find that $\tilde{B}_k = B_k\cap\tilde\Y_j$ is non-empty for $k\in\{1, 2\}$.
Together with the connectedness of $\tilde\Y_j = \tilde{B}_1 \cup \tilde{B}_2$, this implies that $\tilde{B}_1$ and $\tilde{B}_2$ are not disjoint (since closed disjoint sets can be separated by open neighborhoods in metric spaces) and the intersection $\tilde{B}_1 \cap \tilde{B}_2$ hence contains at least one element $y\in\Y_j$.
In particular, there also exist $i_1\in I_1$ and $i_2\in I_2$ such that
\begin{equation*}
  y\in \overline{p_\Y(\Gamma_{i_1j})}\cap
\overline{p_\Y(\Gamma_{i_2j})},
\end{equation*}
which means that $i_1$ and $i_2$ are linked with contact index $j$ and contact point $y$.
We have thus shown that any proper decomposition $I = I_1 \cup I_2$ admits links between the components $I_1$ and $I_2$, implying that all indices in $I$ can be connected by a chain of links.
As discussed above, this makes the Kantorovich potentials $f_a\in\CCc(\mu, \nu)$ almost surely unique and finishes the proof of \cref{thm:uniquedisconnected} under continuity assumption~2.
\end{proof}

\begin{proof}[Proof of \cref{thm:uniquedisconnected} under assumption 1]
The preceding proof has to be adapted to some degree if we work with the slightly stronger continuity requirement that $f^c|_{\supp\,\nu}$ is continuous for each $f^c\in\CCcc(\mu, \nu)$, but in turn do not require any topological features of $\supp\,\nu|_{Y_j}$.
The main difference is that we now allow a contact point $y\in\supp\,\nu$ to be reached along sequences that hop through different components $\Y_j$ (while $j$ was considered fixed for such sequences before).
Thus, we let $\Gamma_i = \supp\,\pi \cap (\X_i \times \Y) = \bigcup_{j\in J} \Gamma_{ij}$ and this time define $i_1,i_2 \in I$ to be linked if there exists a contact point $y\in\supp\,\nu$ such that
\begin{equation}\label{eq:linkedindices1}
  y
  \in
  \overline{\pr_\Y(\Gamma_{i_1})} \cap \overline{\pr_\Y(\Gamma_{i_2})},
\end{equation}
which replaces definition \eqref{eq:linkedindices2}.
In particular, we do not care about the contact index anymore.
It is now easy to check that continuity of $f_a^c|_{\supp\,\nu}$ is sufficient for \textbf{step 1} of the proof above to work as before, and we find that $a_{i_1} - a_{i_2}$ is fixed if $i_1$ and $i_2$ are linked in the sense of \eqref{eq:linkedindices1}.

For \textbf{step 2}, we choose the same approach as above and again exploit the non-degeneracy of $\pi$ to find a suitable index $j\in J_1\cap J_2$ with  $J_1$ and $J_2$ defined as in \eqref{eq:DefinitionIndexsetJ1}.
Then, however, we define the sets 
\begin{equation*}
  B_1 = \Y_j \cap \bigcup_{i\in I_1} \overline{p_\Y(\Gamma_{i})}
  \qquad\text{and}\qquad
  B_2 = \Y_j \cap \bigcup_{i\in I_2} \overline{p_\Y(\Gamma_{i})}
\end{equation*}
somewhat differently, which is better aligned with \eqref{eq:linkedindices1}.
These sets are again closed (use that $I$ is finite) and have positive $\nu$-mass (follows from the definition of $J_1$ and $J_2$).
Furthermore, $\pr_\Y(\supp\,\pi) = \pr_\Y(\bigcup_{i\in I}\Gamma_i)$ is dense in $Y_j$, which leads to $Y_j = B_1 \cup B_2$ along similar lines as in equation \eqref{eq:YjB1B2}.
Connectedness of $Y_j$ thus shows that $B_1\cap B_2$ cannot be empty, from which the existence of $y\in\Y_j$ as well as $i_1\in I_1$ and $i_2\in I_2$ that satisfy \eqref{eq:linkedindices1} follows.
By the same argument as before, the claim of the theorem is established.
\end{proof}

\begin{proof}[Proof of \cref{cor:uniquedisconnected}]
There are two issues that arise in the proof of \cref{thm:uniquedisconnected} when $I$ is allowed to be countable.
The first is that some components $\X_i$ might now have a $\mu$-measure of zero, for which the notion of the $\X_i$-restricted transport problem ceases to make sense.
This can be reconciled by replacing the index set $I$ by $I_+ = \{i\in I\,|\,\mu(\X_i) > 0\}$ throughout the proof.
Then, representation \eqref{eq:uniquedecomposition} of $f_a$ only works on the set $\pr_\X(\supp\,\pi) \cap \bigcup_{i\in I_+} \X_i$, which is, however, sufficient for almost sure uniqueness.

The second issue concerns the sets $B_1$ and $B_2$ constructed in equation \eqref{eq:B1B2}.
For countable $I$, these sets do in general not have to be closed, which would invalidate the ensuing argumentation.
For the two settings described in \cref{cor:uniquedisconnected}, however, this can easily be fixed.
First, if the index set $I^j = \{i\in I\,|\,\pi(\X_i \times Y_j) > 0\}$ has finite cardinality, then we may as well work with the alternative sets
\begin{equation*}
  B_1 = \bigcup_{i\in I_1\cap I^j} \overline{p_\Y(\Gamma_{ij})} \subset Y_j
  \qquad\text{and}\qquad
  B_2 = \bigcup_{i\in I_2\cap I^j} \overline{p_\Y(\Gamma_{ij})} \subset Y_j,
\end{equation*}
for which the remainder of the proof works just as before. Since the unions are finite, these sets are closed.
Secondly, if $Y_j = \{y_j\}$ for $y_j\in Y$ consists of a single point only, then noting that $B_1$ and $B_2$ in \eqref{eq:B1B2} are both non-empty already establishes $B_1 = B_2 = Y_j$, directly yielding the desired contact point $y = y_j$.
In this case, the continuity of $f^c|_{Y_j}$ and the connectedness of $\supp\,\nu|_{Y_j}$ are trivially true.
\end{proof}

\paragraph{Acknowledgements.}
Thomas Staudt and Shayan Hundrieser acknowledge support of DFG RTG 2088 and Axel Munk of DFG CRC 1456.
Axel Munk is supported by, and Thomas Staudt and Shayan Hundrieser were in part funded by, the DFG under Germany's Excellence Strategy EXC 2067/1-390729940.

\printbibliography

@article{montesuma2023recent,
  title={Recent advances in optimal transport for machine learning},
  author={Montesuma, Eduardo Fernandes and Mboula, Fred Ngole and Souloumiac, Antoine},
  journal={arXiv preprint arXiv:2306.16156},
  year={2023}
}

@book{villani2008optimal,
	Author = {Villani, C.},
	Publisher = {Springer},
	Series = {A Series of Comprehensive Studies in Mathematics},
	Title = {Optimal Transport: Old and New},
	Year = {2008}
}

@book{santambrogio2015optimal,
  title={Optimal Transport for Applied Mathematicians: Calculus of Variations, PDEs, and Modeling},
  author={Santambrogio, F.},
  series={Progress in Nonlinear Differential Equations and Their Applications},
  year={2015},
  publisher={Springer}
}

@article{naas2024multimatch,
  title={MultiMatch: geometry-informed colocalization in multi-color super-resolution microscopy},
  author={Naas, Julia and Nies, Giacomo and Li, Housen and Stoldt, Stefan and Schmitzer, Bernhard and Jakobs, Stefan and Munk, Axel},
  journal={Communications Biology},
  volume={7},
  number={1},
  pages={1139},
  year={2024},
  publisher={Nature Publishing Group UK London}
}

@article{hundrieser2024limit,
  title={Limit distributions and sensitivity analysis for empirical entropic optimal transport on countable spaces},
  author={Hundrieser, Shayan and Klatt, Marcel and Munk, Axel},
  journal={The Annals of Applied Probability},
  volume={34},
  number={1B},
  pages={1403--1468},
  year={2024},
  publisher={Institute of Mathematical Statistics}
}

@article{hundrieser2024empirical,
  title={Empirical optimal transport under estimated costs: {Distributional} limits and statistical applications},
  author={Hundrieser, Shayan and Mordant, Gilles and Weitkamp, Christoph A and Munk, Axel},
  journal={Stochastic Processes and their Applications},
  volume={178},
  pages={104462},
  year={2024},
  publisher={Elsevier}
}

@article{hundrieser2024unifying,
  title={A unifying approach to distributional limits for empirical optimal transport},
  author={Hundrieser, Shayan and Klatt, Marcel and Munk, Axel and Staudt, Thomas},
  journal={Bernoulli},
  volume={30},
  number={4},
  pages={2846--2877},
  year={2024},
  publisher={Bernoulli Society for Mathematical Statistics and Probability}
}

@book{rockafellar2015convex,
  title={Convex Analysis},
  author={Rockafellar, Ralph Tyrell},
  year={2015},
  publisher={Princeton University Press}
}

@book{galichon2016optimal,
  title={Optimal Transport Methods in Economics},
  author={Galichon, A.},
  isbn={9781400883592},
  url={https://books.google.de/books?id=GDP9CwAAQBAJ},
  year={2016},
  publisher={Princeton University Press}
}

@article{Schiebinger19,
title = "Optimal-transport analysis of single-cell gene expression identifies developmental trajectories in reprogramming",
journal = "Cell",
volume = "176",
number = "4",
pages = "928 - 943.e22",
year = "2019",
issn = "0092-8674",
doi = "https://doi.org/10.1016/j.cell.2019.01.006",
url = "http://www.sciencedirect.com/science/article/pii/S009286741930039X",
author = "Geoffrey Schiebinger and Jian Shu and Marcin Tabaka and Brian Cleary and Vidya Subramanian and Aryeh Solomon and Joshua Gould and Siyan Liu and Stacie Lin and Peter Berube and Lia Lee and Jenny Chen and Justin Brumbaugh and Philippe Rigollet and Konrad Hochedlinger and Rudolf Jaenisch and Aviv Regev and Eric S. Lander"
}

@article{tameling2021colocalization,
  title={Colocalization for super-resolution microscopy via optimal transport},
  author={Tameling, Carla and Stoldt, Stefan and Stephan, Till and Naas, Julia and Jakobs, Stefan and Munk, Axel},
  journal={Nature Computational Science},
  volume={1},
  number={3},
  pages={199--211},
  year={2021},
  publisher={Nature Publishing Group}
}

@article{Cordero2019Regularity,
title = {Regularity of monotone transport maps between unbounded domains},
journal = {Discrete \& Continuous Dynamical Systems},
volume = {39},
number = {12},
pages = {7101-7112},
year = {2019},
author = {Cordero-Erausquin, Dario and Figalli, Alessio },
}

@Book{figalli2021an,
 author = {Figalli, Alessio and Glaudo, Federico},
 title = {An Invitation to Optimal Transport, Wasserstein Distances, and Gradient Flows},
 publisher = {EMS Press},
 year = {2021},
 address = {Berlin},
 isbn = {3985470103}
 }

@book{kechris2012classical,
  title={Classical Descriptive Set Theory},
  author={Kechris, Alexander},
  volume={156},
  year={2012},
  publisher={Springer Science \& Business Media}
}

@article{klee1968facets,
  title={Facets and vertices of transportation polytopes},
  author={Klee, Victor and Witzgall, Christoph},
  journal={Mathematics of the Decision Sciences},
  volume={1},
  pages={257--282},
  year={1968},
  publisher={American Mathematical Soc.}
}

@article{hung1986degeneracy,
  title={Degeneracy in transportation problems},
  author={Hung, Ming S and Rom, Walter O and Waren, Allan D},
  journal={Discrete Applied Mathematics},
  volume={13},
  number={2-3},
  pages={223--237},
  year={1986},
  publisher={Elsevier}
}

@article{tameling2019empirical,
  title={Empirical optimal transport on countable metric spaces: Distributional limits and statistical applications},
  author={Tameling, Carla and Sommerfeld, Max and Munk, Axel},
  journal={The Annals of Applied Probability},
  volume={29},
  number={5},
  pages={2744--2781},
  year={2019},
  publisher={Institute of Mathematical Statistics}
}

@article{sommerfeld2018inference,
  title={Inference for empirical Wasserstein distances on finite spaces},
  author={Sommerfeld, Max and Munk, Axel},
  journal={Journal of the Royal Statistical Society: Series B (Statistical Methodology)},
  volume={80},
  number={1},
  pages={219--238},
  year={2018},
  publisher={Wiley Online Library}
}

@article{figalli2011local,
  title={Local semiconvexity of Kantorovich potentials on non-compact manifolds},
  author={Figalli, Alessio and Gigli, Nicola},
  journal={ESAIM: Control, Optimisation and Calculus of Variations},
  volume={17},
  number={3},
  pages={648--653},
  year={2011},
  publisher={EDP Sciences}
}

@article{figalli2015partial,
  title={Partial regularity for optimal transport maps},
  author={De Philippis, Guido and Figalli, Alessio},
  journal={Publications Math{\'e}matiques de l'IH{\'E}S},
  volume={121},
  number={1},
  pages={81--112},
  year={2015},
  publisher={Springer}
}

@article{figalli2007existence,
  title={Existence, uniqueness, and regularity of optimal transport maps},
  author={Figalli, Alessio},
  journal={SIAM Journal on Mathematical Analysis},
  volume={39},
  number={1},
  pages={126--137},
  year={2007},
  publisher={SIAM}
}

@article{fathi2010optimal,
  title={Optimal transportation on non-compact manifolds},
  author={Fathi, Albert and Figalli, Alessio},
  journal={Israel Journal of Mathematics},
  volume={175},
  number={1},
  pages={1--59},
  year={2010},
  publisher={Springer}
}

@book{bridson2013metric,
  title={Metric Spaces of Non-Positive Curvature},
  author={Bridson, Martin R and Haefliger, Andr{\'e}},
  volume={319},
  year={2013},
  publisher={Springer Science \& Business Media}
}

@article{brenier1991polar,
  title={Polar factorization and monotone rearrangement of vector-valued functions},
  author={Brenier, Yann},
  journal={Communications on Pure and Applied Mathematics},
  volume={44},
  number={4},
  pages={375--417},
  year={1991},
  publisher={Wiley Online Library}
}

@article{mccann2001polar,
  title={Polar factorization of maps on Riemannian manifolds},
  author={McCann, Robert J},
  journal={Geometric and Functional Analysis},
  volume={11},
  number={3},
  pages={589--608},
  year={2001},
  publisher={Springer}
}

@article{caffarelli1990localization,
  title={A localization property of viscosity solutions to the Monge-Ampere equation and their strict convexity},
  author={Caffarelli, Luis A},
  journal={Annals of Mathematics},
  volume={131},
  number={1},
  pages={129--134},
  year={1990},
  publisher={JSTOR}
}

@article{caffarelli1991some,
  title={Some regularity properties of solutions of Monge Ampere equation},
  author={Caffarelli, Luis A},
  journal={Communications on Pure and Applied Mathematics},
  volume={44},
  number={8-9},
  pages={965--969},
  year={1991},
  publisher={Wiley Online Library}
}

@article{caffarelli1992regularity,
  title={The regularity of mappings with a convex potential},
  author={Caffarelli, Luis A},
  journal={Journal of the American Mathematical Society},
  volume={5},
  number={1},
  pages={99--104},
  year={1992},
  publisher={JSTOR}
}

@article{ma2005regularity,
  title={Regularity of potential functions of the optimal transportation problem},
  author={Ma, Xi-Nan and Trudinger, Neil S and Wang, Xu-Jia},
  journal={Archive for Rational Mechanics and Analysis},
  volume={177},
  number={2},
  pages={151--183},
  year={2005},
  publisher={Springer}
}

@article{loeper2009regularity,
  title={On the regularity of solutions of optimal transportation problems},
  author={Loeper, Gr{\'e}goire},
  journal={Acta Mathematica},
  volume={202},
  number={2},
  pages={241--283},
  year={2009},
  publisher={Institut Mittag-Leffler}
}

@article{ahmad2011optimal,
  title={Optimal transportation, topology and uniqueness},
  author={Ahmad, Najma and Kim, Hwa Kil and McCann, Robert J},
  journal={Bulletin of Mathematical Sciences},
  volume={1},
  number={1},
  pages={13--32},
  year={2011},
  publisher={Springer}
}

@article{mccann2016intrinsic,
  title={The intrinsic dynamics of optimal transport},
  author={McCann, Robert J and Rifford, Ludovic},
  journal={Journal de l’{\'E}cole polytechnique—Math{\'e}matiques},
  volume={3},
  pages={67--98},
  year={2016}
}

@article{moameni2020uniquely,
  title={Uniquely minimizing costs for the Kantorovitch problem},
  author={Moameni, Abbas and Rifford, Ludovic},
  journal={Annales de la Facult{\'e} des Sciences de Toulouse: Math{\'e}matiques},
  volume={29},
  number={3},
  pages={507--563},
  year={2020}
}

@article{lebesgue1905fonctions,
  title={Sur les fonctions repr{\'e}sentables analytiquement},
  author={Lebesgue, Henri},
  journal={Journal de Mathematiques Pures et Appliquees},
  volume={1},
  pages={139--216},
  year={1905}
}

@incollection{kanamori1995emergence,
  title={The emergence of descriptive set theory},
  author={Kanamori, Akihiro},
  booktitle={From Dedekind to G{\"o}del},
  pages={241--262},
  year={1995},
  publisher={Springer}
}

@book{kallenberg1997foundations,
  title={Foundations of Modern Probability},
  author={Kallenberg, Olav},
  volume={2},
  year={1997},
  publisher={Springer}
}

@article{gangbo1996geometry,
  title={The geometry of optimal transportation},
  author={Gangbo, Wilfrid and McCann, Robert J},
  journal={Acta Mathematica},
  volume={177},
  number={2},
  pages={113--161},
  year={1996},
  publisher={Springer}
}

@article{bernton2021entropic,
  title={Entropic optimal transport: {G}eometry and large deviations},
  author={Bernton, Espen and Ghosal, Promit and Nutz, Marcel},
 journal={Duke Mathematical Journal},
  volume={171},
  number={16},
  pages={3363--3400},
  year={2022},
  publisher={Duke University Press}
}

@article{nutz2021entropic,
  title={Entropic optimal transport: Convergence of potentials},
  author={Nutz, Marcel and Wiesel, Johannes},
  journal={Probability Theory and Related Fields},
  volume={184},
  number={1},
  pages={401--424},
  year={2022},
  publisher={Springer}
}

@article{altschuler2021asymptotics,
  title={Asymptotics for semidiscrete entropic optimal transport},
  author={Altschuler, Jason M and Niles-Weed, Jonathan and Stromme, Austin J},
  journal={SIAM Journal on Mathematical Analysis},
  volume={54},
  number={2},
  pages={1718--1741},
  year={2022},
  publisher={SIAM}
}

@article{bercu2021asymptotic,
  title={Asymptotic distribution and convergence rates of stochastic algorithms for entropic optimal transportation between probability measures},
  author={Bercu, Bernard and Bigot, J{\'e}r{\'e}mie},
  journal={The Annals of Statistics},
  volume={49},
  number={2},
  pages={968--987},
  year={2021},
  publisher={Institute of Mathematical Statistics}
}

@article{del2021central,
  title={Central limit theorems for general transportation costs},
  author={Del Barrio, Eustasio and Gonz{\'a}lez-Sanz, Alberto and Loubes, Jean-Michel},
  journal={Annales de l'Institut Henri Poincare (B) Probabilites et statistiques},
  volume={60},
  number={2},
  pages={847--873},
  year={2024},
  organization={Institut Henri Poincar{\'e}}
}

@article{del2019central,
  title={Central limit theorems for empirical transportation cost in general dimension},
  author={Del Barrio, Eustasio and Loubes, Jean-Michel},
  journal={The Annals of Probability},
  volume={47},
  number={2},
  pages={926--951},
  year={2019},
  publisher={Institute of Mathematical Statistics}
}

@article{del2021centralSemidiscrete,
 title={Central limit theorems for semi-discrete Wasserstein distances},
  author={Del Barrio, Eustasio and Gonz{\'a}lez Sanz, Alberto and Loubes, Jean-Michel},
  journal={Bernoulli},
  volume={30},
  number={1},
  pages={554--580},
  year={2024},
  publisher={Bernoulli Society for Mathematical Statistics and Probability}
  }

@incollection{monge,
	author = {Monge, Gaspard},
	booktitle = {Histoire de l'Acad{\'e}mie Royale des Sciences de Paris},
	pages = {666-704},
	title = {M{\'e}moire sur la th{\'e}orie des d{\'e}blais et des remblais},
	year = {1781}
}

@article{kantorovich,
	Author = {Kantorovich, LV},
	Journal = {Doklady Akademii Nauk URSS},
	Pages = {7--8},
	Title = "On the translocation of masses",
	Volume = {37},
	Year = {1942}
}

@book{rachev1998massI,
  title={Mass Transportation Problems: Volume I: Theory},
  author={Rachev, S.T. and R{\"u}schendorf, L.},
%  isbn={9781475780864},
%  url={https://books.google.de/books?id=5U6kngEACAAJ},
  year={1998},
  series = {Probability and Its Applications},
  publisher={Springer}
}

@book{ambrosio2021lectures,
  title={Lectures on Optimal Transport},
  author={Ambrosio, Luigi and Semola, Daniele and Bru{\'e}, Elia},
  year={2021},
  publisher={Springer}
}

@book{panaretos2020invitation,
  title={An Invitation to Statistics in Wasserstein Space},
  author={Panaretos, Victor M and Zemel, Yoav},
  year={2020},
  publisher={Springer Nature}
}

@article{peyre2019computational,
  title={Computational optimal transport: With applications to data science},
  author={Peyr{\'e}, Gabriel and Cuturi, Marco},
  journal={Foundations and Trends{\textregistered} in Machine Learning},
  volume={11},
  number={5-6},
  pages={355--607},
  year={2019},
  publisher={Now Publishers, Inc.}
}

@article{cuesta1989notes,
  title={Notes on the Wasserstein metric in Hilbert spaces},
  author={Cuesta, Juan Antonio and Matr{\'a}n, Carlos},
  journal={The Annals of Probability},
  pages={1264--1276},
  year={1989},
  publisher={JSTOR}
}

@article{smith1987note,
  title={Note on the optimal transportation of distributions},
  author={Smith, Cyril S and Knott, Martin},
  journal={Journal of Optimization Theory and Applications},
  volume={52},
  number={2},
  pages={323--329},
  year={1987},
  publisher={Springer}
}

@article{bertrand2008existence,
  title={Existence and uniqueness of optimal maps on Alexandrov spaces},
  author={Bertrand, J{\'e}r{\^o}me},
  journal={Advances in Mathematics},
  volume={219},
  number={3},
  pages={838--851},
  year={2008},
  publisher={Elsevier}
}

@article{gigli2012optimal,
  title={Optimal maps in non branching spaces with Ricci curvature bounded from below},
  author={Gigli, Nicola and others},
  journal={Geometric and Functional Analysis},
  volume={22},
  number={4},
  pages={990--999},
  year={2012}
}

@article{ambrosio2014slopes,
  title={Slopes of Kantorovich potentials and existence of optimal transport maps in metric measure spaces},
  author={Ambrosio, Luigi and Rajala, Tapio},
  journal={Annali di Matematica Pura ed Applicata},
  volume={193},
  number={1},
  pages={71--87},
  year={2014},
  publisher={Springer}
}

@book{federer2014geometric,
  title={Geometric Measure Theory},
  author={Federer, Herbert},
  year={2014},
  publisher={Springer}
}

@article{beiglbock2011duality,
  title={Duality for Borel measurable cost functions},
  author={Beiglb{\"o}ck, Mathias and Schachermayer, Walter},
  journal={Transactions of the American Mathematical Society},
  volume={363},
  number={8},
  pages={4203--4224},
  year={2011}
}

@article{ruschendorf2007monge,
  title={Monge-Kantorovich transportation problem and optimal couplings},
  author={R{\"u}schendorf, Ludger},
  journal={Jahresbericht der Deutschen Mathematiker Vereinigung},
  volume={109},
  number = {3},
  pages={113--137},
  year={2007}
}

@article{qi1989maximal,
  title={The maximal normal operator space and integration of subdifferentials of nonconvex functions},
  author={Qi, Liqun},
  journal={Nonlinear Analysis: Theory, Methods \& Applications},
  volume={13},
  number={9},
  pages={1003--1011},
  year={1989},
  publisher={Elsevier}
}

@Book{ambrosio2006,
  author    = {Ambrosio, Luigi and Gigli, Nicola and Savare, Giuseppe},
  publisher = {Birkhäuser Basel},
  title     = {Gradient Flows: In Metric Spaces and in the Space of Probability Measures},
  year      = {2006},
  isbn      = {9783764373092},
  pages     = {333},
}

@Article{yang2023,
  author    = {Yunan Yang and Levon Nurbekyan and Elisa Negrini and Robert Martin and Mirjeta Pasha},
  journal   = {{SIAM} Journal on Applied Dynamical Systems},
  title     = {Optimal Transport for Parameter Identification of Chaotic Dynamics via Invariant Measures},
  year      = {2023},
  month     = {feb},
  number    = {1},
  pages     = {269--310},
  volume    = {22},
  doi       = {10.1137/21m1421337},
  publisher = {Society for Industrial {\&} Applied Mathematics ({SIAM})},
}

\clearpage
\begin{appendix}

\section{Auxiliary results and omitted proofs}
\label{app:proofs}

In a brief assertion of his \citeyear{lebesgue1905fonctions} memoir, Henry Lebesgue famously sketched an erroneous proof claiming that the projection of Borel subsets of the plane $\RR^2$ onto one of the coordinate axes is again Borel.
The invalidity of this claim was uncovered in 1916 by Mikhail Suslin, which inspired the study of what is now called \emph{analytic sets} or \emph{Suslin sets}
\parencite{kanamori1995emergence}.
Placed into the context of our work, we learn that it can happen that projected sets of the form $\pr_\X(A)$ and $\pr_Y(A)$ for Borel sets $A\subset\XY$ are not Borel again.
However, these sets are analytic and thus \emph{universally measurable}.
In particular, they can be approximated from within and without by Borel sets whose difference is a null set.
This settles potential measurability issues in a satisfactory manner.

\begin{lemma}{}{measurability}
  Let $\X$ and $\Y$ be Polish and $\pi\in\couple(\mu, \nu)$ be a transport plan between $\mu\in\PC(\X)$ and $\nu\in\PC(\Y)$.
  Suppose $A\subset\X\times\Y$ is Borel with $\pi(A) = 1$.
  Then there exist Borel sets
  \begin{equation*}
    A_\mu \subset \pr_\X(A)
    \qquad\text{and}\qquad
    A_\nu \subset \pr_\Y(A)
  \end{equation*}
  such that $\mu(A_\mu) = \nu(A_\nu) = 1$.
  In particular, $A_\mu$ and $A_\nu$ are dense in $\supp\,\mu$ and $\supp\,\nu$.
\end{lemma}

\begin{proof}[Proof of \cref{lem:measurability}]
  We only show the statement for $\mu$ since the one for $\nu$ follows equivalently.
  According to \cite[Exercise~14.3]{kechris2012classical}, the set $\pr_\X(A) \subset \X$ is \emph{analytic}, i.e., the continuous image of a Polish space.
  By \cite[Theorem~21.10]{kechris2012classical}, every analytic set is \emph{universally measurable} (see Definition~12.5 for the term \emph{standard Borel space}), which implies that $\pr_\X(A) = A_\mu \cup N$ where $A_\mu\subset\X$ is Borel and $N$ is a subset of a Borel $\mu$-null set $M\subset\X$ (see \cite[Section~17.A]{kechris2012classical}, for respective definitions).
  It is left to show that $\mu(A_\mu) = \mu(A_\mu\cup M) = 1$, which follows from observing that $A \subset \pr_\X(A) \times \Y \subset (A_\mu\cup M) \times \Y$ and thus
  \begin{equation*}
    \mu(A_\mu\cup M) = \pi\big((A_\mu\cup M)\times \Y\big) \ge \pi(A) = 1.
    \qedhere
  \end{equation*}
\end{proof}

Most of the time, we employ \cref{lem:measurability} with the choice $A = \supp\,\pi$ for an optimal transport plan $\pi$, making sure that properties on the sets $\pr_\X(\supp\,\pi)$ and $\pr_\Y(\supp\,\pi)$ are valid $\mu$- and $\nu$-almost surely.
We next highlight a simple consequence of \cref{lem:regularitycompact}, relating the points $x\in\supp\,\mu$ with $x\not\in\pr_\X(\supp\,\pi)$ to the ones where $\pi$ does not induce regularity.

\begin{lemma}{}{notpartiallycompact}
  Let $X$ and $Y$ be Polish, $\mu\in\PC(\X)$, $\nu\in\PC(\Y)$, and $c\colon\XY\to\RR_+$ continuous such that $\ot_c(\mu, \nu) < \infty$.
  Let $\pi\in\couple(\mu, \nu)$ be an optimal transport plan. Then
  \begin{equation*}
    \closure\big(\supp\,\mu \setminus \pr_\X(\supp\,\pi)\big)
    \subset
    \Sigma_\pi
    =
    \big\{x\in\supp\,\mu\,\big|\,\pi~\text{does not induce regularity at}~x\big\}
  \end{equation*}
  and the set $\Sigma_\pi$ is closed in $\supp\,\mu$.
\end{lemma}

\begin{proof}
  To see that $\Sigma_\pi$ is closed, assume $\pi$ to induce regularity at $x\in\supp\,\mu$ with relatively open $U\subset\supp(\mu)$ and compact $K\subset\Y$.
  Then $\pi$ is also inducing regularity at any other $x'\in U$ with the same $U$ and $K$, showing that $\supp\,\mu\setminus\Sigma_\pi$ is open.
  It now suffices to note that $\supp\,\mu\setminus\Sigma_\pi \subset\pr_\X(\supp\,\pi)$ according to \cref{lem:regularitycompact}, which, together with closedness of $\Sigma_\pi$, implies the inclusion.
\end{proof}

We now turn to \cref{lem:fullrestriction}, which states that Kantorovich potentials behave as expected when restricted to subsets $\Xres\subset\X$ and $\Yres\subset\Y$ with full $\mu$- and $\nu$-mass.
The proof of this statement follows the reasoning behind \cite[Lemma~5.18 and Theorem~5.19]{villani2008optimal}.
Cases of particular interest to us are restrictions to the (interior of the) support of $\mu$ or $\nu$ (if the boundary carries no mass).
Then the subsets $\Xres$ and $\Yres$ are either closed or open, and so they are always Borel and Polish.

\begin{proof}[Proof of \cref{lem:fullrestriction}]
  By assumption, $\Xres\subset\X$, $\Yres\subset\Y$, and $\XYres\subset\XY$ are Borel and Polish subsets with $\mu(\Xres) = \nu(\Yres) = 1$.
  Since the Borel $\sigma$-algebra of a restricted spaces coincides with the respective subspace $\sigma$-algebra (see, e.g., \cite[Lemma~1.6]{kallenberg1997foundations}), it is easy to recognize that $T_c(\mu, \nu) = T_\cres(\mu, \nu)$ with equal optimal plans $\pi\in\CC(\mu, \nu)$, where we permissively identify the measures $\mu$, $\nu$, and $\pi$ with their restrictions to the Borel sets $\Xres$, $\Yres$, and $\XYres$ of full mass.

  Recall that $\tilde\osupp = \supp\,\pi \cap (\XYres)$ and observe $\pi(\tilde\osupp) = 1$, which implies that $\tilde\osupp$ is dense in the support of $\pi$.
  We begin with the claim on restrictions. Let $f\in \CCc(\mu, \nu)$ and define the $\cres$-concave function $\fres\colon\Xres\to\RRext$ via
  \begin{equation*}
    \fres(x)
    =
    \big(f^c|_{\Yres}\big)^{\cres}(x)
    =
    \inf_{y\in\Yres} c(x, y) - f^c(y).
  \end{equation*}
  For any $(x_0, y_0) \in \tilde\osupp$, we calculate
  \begin{align*}
    \fres(x_0)
    &=
    \inf_{y\in\Yres} c(x_0, y) - f^c(y) 
    \le
    c(x_0, y_0) - f^c(y_0)
    =
    f(x_0), \\
    \fres(x_0)
    &=
    \inf_{y\in\Yres} c(x_0, y) - \inf_{x\in\X} c(x, y) + f(x)
    \ge
    f(x_0),
  \end{align*}
  which shows that $\fres = f$ on $\pr_\X(\tilde\osupp)$. Similarly,
  \begin{align*}
    \fres^{\cres}(y_0)
    &=
    \inf_{x\in\Xres} c(x, y_0) - \fres(x)
    \le
    c(x_0, y_0) - f(x_0)
    =
    f^c(y_0), \\
    \fres^{\cres}(y_0)
    &=
    \inf_{x\in\Xres} c(x, y_0) - \inf_{y\in\Yres} c(x, y) + f^c(y)
    \ge
    f^c(y_0),
  \end{align*}
  which asserts $\fres^{\cres} = f^c$ on $\pr_\Y(\tilde\osupp)$.
  Since the set $\tilde\Gamma$ has full $\pi$-measure, it follows that
    $ \OT_{\cres}(\mu, \nu)
    =
    \OT_c(\mu, \nu)
    =
    \pi\,(f \oplus f^c)
    =
    \pi\,(\tilde{f} \oplus \tilde{f}^\cres)
   $.
  This shows $\fres\in \CC_\cres(\mu, \nu)$ and thus proves the first statement.

  We next turn towards the claim on extending potentials, where we begin with $\fres\in \CC_\cres(\mu, \nu)$ and observe $\tilde\osupp \subset \partial_{\cres}\fres$ (which is true since $\tilde\osupp$ equals the support of the measure $\pi$ restricted to $\XYres$).
  We extend $\fres$ to a function $f\colon\X\to\RR$ on all of $\X$ via defining
  \begin{equation*}
    f(x) = \inf_{y\in\Yres} c(x, y) - \fres^\cres(y).
  \end{equation*}
  This function is $c$-concave, since it is the $c$-transform of a function that equals $\fres^\cres$ on $\Yres$ and $-\infty$ on $\Y\setminus\Yres$.
  Furthermore, since $\fres$ is $\cres$-concave, the definition of $f$ directly shows that it coincides with $\fres$ ($= \fres^{\cres\cres}$) on $\Xres$.
  For $(x_0, y_0)\in\tilde\osupp$, we furthermore note that
  \begin{align*}
    f^c(y_0)
    &=
    \inf_{x\in\X} c(x, y_0) - f(x)
    \le
    c(x_0, y_0) - f(x_0) = \fres^\cres(y_0), \\
    f^c(y_0)
    &=
    \inf_{x\in\X} c(x, y_0) - \inf_{y\in\Yres} c(x, y) + \fres^{\cres}(y)
    \ge
    \fres^c(y_0),
  \end{align*}
  and thus $f^c = \fres^\cres$ on $\pr_Y(\tilde\osupp)$. The optimality of
  $f$ is checked like above, yielding $f\in\CCc(\mu, \nu)$.
\end{proof}

Note that the sets of consensus $\pr_\X(\tilde\osupp)$ and $\pr_\Y(\tilde\osupp)$ in \cref{lem:fullrestriction} can be enlarged to the (potentially slightly bigger) sets $\pr_\X(\supp\,\pi)\cap\Xres$ and $\pr_\Y(\supp\,\pi)\cap\Yres$ when restricting Kantorovich potentials.
To prove this, the density of $\tilde\osupp$ in $\supp\,\pi$ can be exploited in combination with \cref{lem:regularity}.
For extending a potential $\fres$, the proof above shows that it is always possible to pick an extension that agrees with $\fres$ on all of $\Xres$, or alternatively to pick an extension whose $c$-transform agrees with $\fres^\cres$ on all of $\Yres$.

We next prove the claim stated in the context of equation \eqref{eq:strictlyseparatedcosts}, which provides an example where degeneracy of the optimal transport plan leads to non-unique Kantorovich potentials.

\begin{lemma}{}{ambiguous}
  Let $X = Y$ be Polish, $\mu = \nu \in \PC(\X)$ with $\supp(\mu) = \X_1 \cup \X_2$, and $c\colon\X^2\to\RR_+$ be continuous and symmetric with $c(x, x) = 0$ for all $x\in\X$.
  If
  \begin{equation*}
      \Delta = \inf_{x_1\in\X_1, x_2\in\X_2} c(x_1, x_2) > 0,
  \end{equation*}
  then for all $a,b\in\RR$ with $|a-b|\le \Delta$, there exists $f_{a,b}\in\CCc(\mu, \mu)$ such that $f_{a,b} = a$ on $\X_1$ and $f_{a,b} = b$ on $\X_2$.
\end{lemma}

\begin{proof}
  It is apparent that $\OT_c(\mu, \mu) = 0$.
  For real numbers $a,b\in \mathbb{R}$ with $| a-b|\leq \Delta$, we define the map $g\colon\X \to \RR\cup\{-\infty\}$ via
  \begin{equation*}
    g(x) =
    \begin{cases}
      -a & \text{if } x \in \X_1,\\
      -b & \text{if } x \in \X_2,\\
      -\infty & \text{if } x \notin \X_1\cup \X_2.
    \end{cases}
  \end{equation*}
  For $x \in \X_1$, the $c$-concave function $f_{a,b}
  \coloneqq g^{c}$ fulfills
  \begin{equation*}
      f_{a,b}(x)
      = \inf_{y \in \X } c(x,y) - g(y) = \inf_{y\in \X} \begin{cases}
       a & \text{ for } y = x\in \X_1,\\
       c(x,y)+ a & \text{ for } y\in \X_1,\\
       c(x,y)+ b & \text{ for } y\in \X_2, \\
       \infty & \text{ for } y\notin \X_1 \cup \X_2. \\
      \end{cases}
  \end{equation*}
  Since $c \ge 0$ and $a \le \Delta + b \le c(x, y) + b$ for $x\in\X_1, y\in\X_2$, we find $f_{a,b} = a$ on $\X_1$.
  Likewise, we also find $f_{a,b} = b$ on $\X_2$.
  By a similar argument, it follows that $f_{a,b}^c \le -a$ on $\X_1$ and $\smash{f_{a,b}^c} \le -b$ on $\X_2$.
  Due to the lower bound $g \le g^{cc} = \smash{f_{a,b}^c}$ (see \cite[Proposition~5.8]{villani2008optimal}), we conclude that $f^c_{a, b} = -f_{a, b}$ on $\X_1\cup\X_2$.
  This implies $\OT_c(\mu, \mu) = 0 = \pi\,\big(f_{a,b} \oplus f_{a,b}^c\big)$ for any optimal $\pi$, asserting that $f_{a,b}$ is indeed a Kantorovich potential.
\end{proof}

Next, we address the claim raised in \cref{lem:lipschitzgradients}, which states that (locally) Lipschitz continuous functions on connected manifolds coincide (up to constants) if their gradients coincide almost surely (in charts).
This is a simple generalization of corresponding results on $\RR^d$, which can, for example, be gathered from considerations in \textcite{qi1989maximal}.

\begin{proof}[Proof of \cref{lem:lipschitzgradients}]
  Let $f_1, f_2\colon M \to \RR$ be locally Lipschitz on a $d$-dimensional smooth manifold $M$, and let $(\varphi_x)_{x\in M}$ be a family of charts with $x\in U_x = \domain\,\varphi_x$ for all $x\in M$.
  By translating, restricting, and rescaling the charts, we can assume that $\range\,\varphi_x = B_1 \subset \RR^d$ is the unit ball and that $f_{i,x} = f_i\circ\varphi_x^{-1}\colon B_1 \to \RR$ is Lipschitz for each $x\in M$ and $i \in \{1, 2\}$.
  Since $\nabla f_{1,x} = \nabla f_{2,x}$ holds by assumption on a set with full Lebesgue measure, we can conclude (e.g., by formula~(2) of \cite{qi1989maximal}) that $f_{1,x} - f_{2,x} = c_x$ on all of $B_1$, where $c_x\in\RR$ is a constant.
  This implies $f_1|_{U_x} - f_2|_{U_x} = c_x$ as well.
  To see that $c_x$ is actually independent of $x$, note that $c_{x'} = c_x$ holds for any $x'\in U_x$, since then $U_x\cap U_{x'} \neq 0$.
  Thus, $x\mapsto c_x$ is a locally constant function.
  The connectedness of $M$ implies that it is also constant on the whole space.
  To see this, just note that the set $V = \{x\in\X\,|\,c_x = c_{x_0}\} \neq \emptyset$ and its complement are both open (for some fixed $x_0\in\X$), which makes $V$ open and closed.
  Hence, $V = M$ and the claim $f_1 - f_2 = c_{x_0}$ is established.
\end{proof}

We now turn to \cref{rem:semiconcave}, where it is claimed that \cref{cor:uniqueconnectedlipschitz} is still true if $c(\cdot, y)$ is assumed to be locally semiconcave uniformly in $y\in\Y$ (instead of locally Lipschitz uniformly in $y$).

\begin{proof}[Proof of \cref{rem:semiconcave}]
  We adhere to the general proof strategy of \cref{thm:uniqueconnected}, but will provide additional details for some of the arguments.
  For every $x\in\X$, let $\varphi_x \colon U_x \to V_x \subset \RR^d$ denote a chart around $x$.
  By restricting and translating, we can assume that $V_x$ is convex with $\varphi(x) = 0\in V_x$.
  Recalling equation~\eqref{eq:semiconcave}, we may also assume that $v \mapsto c\big(\varphi^{-1}_x(v), y\big) - \lambda_x \|v\|^2$ is concave on $V_x$ for some $\lambda_x > 0$ and all $y\in\Y$.
  Due to the nature of $c$-conjugates as infima, we find that the function $v \mapsto h_x(v) = f\big(\varphi^{-1}_x(v)\big) - \lambda_x \|v\|^2$ is concave as well for any $f\in\CCc(\mu, \nu)$ \parencite[Theorem~5.5]{rockafellar2015convex}.

  We first show that $f > -\infty$ on $M = \interior(\supp\,\mu)$.
  Indeed, assume that there existed a point $x\in M$ with $f(x) = -\infty$.
  Then $h_x(0) = -\infty$, and, since the effective domain $h_x^{-1}(\RR)\subset V_x$ is convex, it followed that $h_x^{-1}\big(\{-\infty\}\big)$ contained at least an open half-space (intersected with $V_x$) touching the origin.
  Hence, $f = -\infty$ would have to hold on an open subset of $M \subset \supp\,\mu$, which cannot be true, since $\pr_\X(\supp\,\pi)$ is dense in $\supp\,\mu$ and $f > -\infty$ on $\pr_\X(\supp\,\pi)$ due to $\supp\,\pi \subset \partial_c f$.
  Since (locally semi-)concave functions are locally Lipschitz in the interior of their effective domain \parencite[Theorem~10.4]{rockafellar2015convex}, we can conclude that each Kantorovich potential is locally Lipschitz on all of $M$.
  
  Next, since $\Gamma = \supp\,\pi$, we find a Borel set $A \subset \pr_\X(\Gamma) \subset \X$ that satisfies $\mu(A) = 1$ (\cref{lem:measurability}).
  Hence, $B = \range\,\varphi \setminus \varphi(A \cap \domain\,\varphi)$ is a $\varphi_\#\mu$-null set for any chart $\varphi$ of $M$.
  Due to condition~\eqref{eq:absolutecontinuity}, it is also a Lebesgue-null set and we conclude that $\pr_\X(\Gamma)$ has full Lebesgue measure in charts of $M$.
  We can now argue as in \cref{thm:uniqueconnected} to find that any two Kantorovich potentials $f_1, f_2\in\CCc(\mu, \nu)$ coincide on $M$. 
  It remains to show that $f_1 = f_2$ holds on the boundary of $\supp\,\mu$ as well.
  Let $x\in\partial\,\supp\,\mu$ and let $(x_n)_{n\in\NN} \subset M$ be a sequence converging to $x$.
  Then it holds that $\lim_{n\to\infty} f_i(x_n) = f_i(x)$ for $i\in\{1, 2\}$.
  This can be seen as follows: if $f_i(x) > -\infty$, then the limit holds since (locally semi-)concave functions are continuous on their effective domain \parencite[Theorem~10.1]{rockafellar2015convex}, and if $f_i(x) = -\infty$, then the limit holds due to upper-semicontinuity of $f_i$.
  Since $f_1(x_n) = f_2(x_n)$ for all $n$, this establishes equality of $f_1$ and $f_2$ on the full support.
\end{proof}

Finally, to conclude this appendix, we prove that the various conditions established in \cref{lem:assumptionG} imply assumption~(G) required by \cref{thm:euclideaninteriorregularity}.

\begin{proof}[Proof of Lemma~\ref{lem:assumptionG}]
  We first verify that condition 3 implies (G).
  Let $g(a) = \inf_{\|u\| = 1, b \ge a} h'_u(b)$.
  This function is non-decreasing and satisfies $g(a) \to \infty$ as $a\to\infty$ by assumption.
  Let $0 < t_a < 1$ be such that $t_a\to 0$ and $t_a\,g(a - 1) \to \infty$ as $a\to\infty$.
  Set $\delta_x = - t_{\|x\|} \hat{x}$, where $\hat{x} = x/\|x\|$ for any $x\in\RR^d\setminus\{0\}$.
  Then $\delta_x\to 0$ as well as
  \begin{equation*}
    h(x) - h(x + \delta_x)
    =
    h_{\hat{x}}\big(\|x\|\big) - h_{\hat{x}}\big(\|x\| - t_{\|x\|}\big)
    =
    \int_{-t_{\|x\|}}^0 \!\!\! h'_{\hat{x}}\big(\|x\| + t\big) \dif t
    \ge
    t_{\|x\|} g\big(\|x\| - 1\big)
    \to
    \infty
  \end{equation*}
  as $x\to\infty$, which establishes (G).
  Next, it is straightforward to show that condition 2 implies condition 3 and thus (G), as the infimum over directions $u$ becomes immaterial in this case.
  Finally, assume condition 1, i.e., that $h$ is convex and satisfies the superlinear growth condition $\lim_{x\to\infty} h(x)/ \|x\| = \infty$.
  We demonstrate that this implies condition 3 as well.
  For if it violated condition 3, there would be directions $u_n$ and values $b_n > 0$ such that $b_n \to \infty$ as $n\to\infty$ while $h'_{u_n}(b_n)$ is well-defined and uniformly bounded by some $K > 0$.
  Since $h'_u$ is monotonically increasing for every direction $u$ due to the convexity of $h_u$, we conclude
  \begin{equation*}
    h(b_n u_n) = h_{u_n}(b_n)
    \le
    h(0) + b_n h'_{u_n}(b_n)
    \le
    h(0) + b_n K.
  \end{equation*}
  This contradicts the superlinear growth property.
\end{proof}

\end{appendix}

\end{document}